\documentclass[a4paper,12pt]{amsart}
\usepackage[utf8]{inputenc}
\usepackage{extsizes}
\usepackage[T1,T2A]{fontenc}
\usepackage[french,russian,english]{babel}
\usepackage{amsmath,amssymb,amsfonts,amsthm}
\usepackage[unicode]{hyperref}
\usepackage{graphicx}
\usepackage[backgroundcolor=white]{todonotes}

\usepackage{tikz}
\usepackage[backend=biber,firstinits,autolang=langname]{biblatex}
\usetikzlibrary{arrows}
\expandafter\addto\csname HyLang@english\endcsname{

}

\newcounter{thm}
\newtheorem{theorem}[thm]{Theorem}
\newtheorem*{theorem*}{Theorem}
\newtheorem*{maintheoremThm}{Main Theorem}

\newtheorem{lemma}[thm]{Lemma}

\newtheorem*{conjecture*}{Conjecture}

\theoremstyle{remark}
\newtheorem{remark}[thm]{Remark}
\newtheorem*{remark*}{Remark}

\theoremstyle{definition}
\newtheorem*{definition*}{Definition}
\newtheorem{definition}[thm]{Definition}

\renewcommand{\Re}{\mathop{\mathrm{Re}}}
\renewcommand{\Im}{\mathop{\mathrm{Im}}}
\newcommand{\eps}{\varepsilon}
\newcommand{\bbC}{\mathbb C}
\newcommand{\bbR}{\mathbb R}
\newcommand{\bbQ}{\mathbb Q}
\newcommand{\bbH}{\mathbb H}
\newcommand{\bbN}{\mathbb N}
\newcommand{\bbZ}{\mathbb Z}
\newcommand{\rot}{\mathrm {rot}\,}

\DeclareMathOperator{\dist}{dist}
\title{Self-similarity of bubbles}

\author{Nataliya Goncharuk}
\email{ng432@cornell.edu}
\address{%
Cornell University\\
Department of Mathematics\\
310 Malott Hall\\
Ithaca, NY 14853-4201 USA}
\thanks{Supported by RFBR project 16-01-00748-a and Laboratory Poncelet.}
\thanks{Cornell University, Department of Mathematics}
\makeatletter
\hypersetup{
  pdftitle=\@title,
  pdfauthor=\authors,
  pdfkeywords=\@keywords,
  pdfsubject=\@subjclass
}
\makeatother
\begin{document}

\begin{abstract}
Bubbles is a fractal-like set related to a circle diffeomorphism; they are a complex analogue to Arnold tongues. In this article, we prove an approximate self-similarity of bubbles.
\end{abstract}
\maketitle

\section{Introduction}

\subsection{Complex rotation numbers. Arnold's construction}
In what follows,  $f\colon \bbR/\bbZ \to \bbR/\bbZ$ is an analytic orientation-preserving circle diffeomorphism. Let $F\colon \bbR\to\bbR$ be a lift of $f$ to the real line.

In 1978, V.Arnold \cite[Sec. 27]{Arn-english} suggested the following construction. Given $\omega\in \bbH$ and a small positive $\eps \in\bbR$, consider the strip
$$
   \Pi^{\eps}:=\{z \in \bbC \mid -\eps <\Im z < \Im\omega+\eps\}.\\
$$
Extend $F$ analytically to the $\eps$-neighborhood of the real axis; the analytic extensions of $f$ and $F$  are still denoted by $f$ and $F$ respectively. Put
$$E(F+\omega):= \Pi^{\eps} / (z\sim z+1, z \sim F(z)+\omega).
$$
For a small $\eps$, the quotient space $E(F+\omega)$ is a torus, it inherits the complex structure from $\bbC$ and does not depend on $\eps$.

The complex torus $E(F+\omega)$ has two naturally distinguished generators of the first homology group, namely
$\bbR/\bbZ$ and the class of $[0, F(0)+\omega]$. Thus  the modulus  of $E(F+\omega)$ is well-defined: for a unique   $\tau$ in the upper half-plane, there exists a biholomorphism
\begin{equation}
\label{eq-Hw}
 H_{\omega} \colon E(F+{\omega}) \to \bbC / (\bbZ+\tau \bbZ)
\end{equation}
that takes the first generator of $E(F+\omega)$ to   the class of $\bbR/\bbZ$, and the second generator to the class of $\tau\bbR/\tau\bbZ$.

Here and below $\bbH, \bbH\subset \bbC$, stands for the open upper half-plane.
\begin{definition}

The   modulus of the complex torus $E(F+{\omega})$ is called the \emph{complex rotation number} of $F+{\omega}$ and denoted by $\tau({F+\omega}):=\tau \in \bbH$.

\end{definition}
 We  also use the notation $\tau_F(\omega) := \tau(F+\omega)$ if we want to stress the dependence on $\omega$.
% The map $\omega\mapsto \tau(F+\omega)$ commutes with the shift by one, so induces a self-map of $\bbC/\bbZ$. This self-map is denoted by $\omega\mapsto \tau(f+\omega)$.
% In \ref{XB_NG}, the autrors introduce the number $\tau(f+\omega) = \tau(F+\omega) \mod 1$, $\tau(f+\omega)\in \bbH/\bbZ$, is well-defined. In other words, the map $\omega\mapsto \tau(F+\omega)$ is a lift of the map $\omega \mapsto \tau(f+\omega)$ to $\bbC$.

% In initial Arnold's construction, $\omega$ was supposed to be purely imaginary. The above version of this construction was suggested by R.Fedorov.
The term ``complex rotation number'' is due to E. Risler, \cite{Ris}.
Complex rotation number $\tau(F+\omega)$ depends holomorphically on $\omega\in \bbH$, see \cite[Sec. 2.1, Proposition 2]{Ris}.
\subsection{Dependence on the lift}
\label{sec-lift}
The second generator of $E(F+\omega)$, and thus the complex rotation number $\tau(F+\omega)$, depends on the choice of a lift $F$ of the circle diffeomorphism $f$. Namely, $\tau(F+\omega+1) =\tau(F+\omega)+1$. So the class of $\tau(F+\omega)$ in $\bbH/\bbZ$ depends on $f$ and  $\omega\in\bbH/\bbZ$ only. We denote it by $$\tau(f+\omega):=\tau(F+\omega) \mod 1,\quad \tau(f+\omega) \in \bbH/\bbZ.$$ We also use the notation $\tau_f(\omega):=\tau(f+\omega)$. Clearly, $\tau_F$ is a lift of $\tau_f\colon \bbH/\bbZ \to \bbH/\bbZ$ to $\bbH$.

\subsection{Rotation numbers of circle diffeomorphisms}
Here we list some well-known facts about rotation numbers, see \cite[Sections 3.11, 3.12]{KaHa} for the proofs.
\begin{definition}
For a circle diffeomorphism $f$, let $F$ be a lift of $f$ to the real line. The \emph{rotation  number} of $F$ is
 \begin{equation}
 \label{eq-rot}
 \rot F := \lim_{n\to \infty} \frac {F^n(x)}{n}.
 \end{equation}
 The limit in \eqref{eq-rot} exists and does not depend on $x$.
 The rotation number of a circle diffeomorphism $f$ is $\rot f :=\rot F \mod 1, \rot f \in \bbR/\bbZ$; it does not depend on the choice of the lift.
\end{definition}
The rotation number is invariant under continuous conjugacies; it is rational if and only if $f$ has periodic orbits. If $\rot f $ is irrational, then all the orbits of $f$ on the circle are ordered in the same way as the orbits of the irrational rotation $T_{\rot f} \colon x \mapsto x+\rot f$ on the circle.

\subsection{Extension of the complex rotation number to the real axis}
 Recall that a periodic orbit of a circle diffeomorphism is called \emph{parabolic} if its multiplier is one, and \emph{hyperbolic} otherwise. A diffeomorphism with periodic orbits such that all of them are hyperbolic is called \emph{hyperbolic}.

The question due to Arnold (see \cite[Sec. 27]{Arn-english}) was to investigate the complex rotation number $\tau(F+\omega)$ as $\omega$ approaches the real axis. He conjectured that if for a real $\omega_0$, $\rot(F+\omega_0)$ is Diophantine, then $\lim_{\omega\to \omega_0}\tau(F+\omega) = \rot(F+\omega_0)$. Two independent proofs of this conjecture were given in \cite{Ris} and \cite{M-en}. This statement does not hold if $f+{\omega_0}$ is hyperbolic, as was proved in \cite{YuIM}; this result was strengthened in \cite{NG2012-en}. The case of a diffeomorphism with parabolic cycles  was studied by J.Lacroix (unpublished) and in \cite{NG2012-en}.

The following result gives the description of the limit behaviour of $\tau_f$ near the real line, including  the case when $\rot(f+\omega_0)$ is a Liouville number.  The analyticity of $\bar \tau_f$ in the last subcase below is not included in  \cite[Main theorem]{XB_NG}, but it constitutes  \cite[Theorem 1.2]{NG2012-en}, see also \cite[Theorem 2]{XB_NG}.

\begin{theorem}[X. Buff, N. Goncharuk \cite{XB_NG}]
\label{th-XB_NG}
Let $f\colon \bbR/\bbZ \to \bbR/\bbZ$ be an orientation-preserving analytic circle diffeomorphism.  Then the holomorphic function $\tau_f \colon \bbH/\bbZ \to \bbH/\bbZ $ has a continuous extension $\bar \tau_f $ to $ \overline{\bbH}/\bbZ $. Assume $\omega \in \bbR/\bbZ $.
\begin{enumerate}
 \item If $\rot (f+{\omega})$ is irrational, then $\bar \tau_f(\omega)=\rot(f+{\omega})$.
 \item If $\rot (f+{\omega})$ is rational and $f+{\omega}$ has a parabolic periodic orbit, then again $\bar \tau_f(\omega)=\rot (f+{\omega})$.
 \item If $f+{\omega}$ is hyperbolic with rotation number $p/q$  on an open interval  $\omega\in I\subset \bbR/\bbZ $, then $\bar \tau_f$ is analytic on $I$ and $\bar \tau_f(\omega)\in \bbH/\bbZ $ for $\omega \in I$.

 Moreover, $\bar \tau_f(\omega)$ belongs to the closed disk of radius $D_f/(4\pi q^2)$ tangent to $\bbR/\bbZ$ at $p/q$. Here $D_{f} := \int_{\bbR/\bbZ} \left|\frac{f''(x)}{f'(x)}\right| dx$ is the distortion of $f$, and $p/q$ is assumed to be irreducible.
\end{enumerate}
\end{theorem}
The number $\bar \tau_f(\omega)$ depends on $f+\omega$ only. In the first two cases of Theorem \ref{th-XB_NG}, this follows from $\bar \tau_f(\omega) = \rot (f+\omega)$. In the third case, this follows from the direct construction of $\bar \tau_f$ presented in Sec. \ref{sec-Buff} and in \cite[Sec. 5]{XB_NG}, see Remark~\ref{rem-depend-on-g+omega} below. This motivates the following notation: $\bar \tau (f+\omega):=\bar \tau_f(\omega)$.

The value of $\bar \tau(f+{\omega})$ at $\omega \in \bbR/\bbZ$ is also called the complex rotation number of $f+\omega$.
So each hyperbolic circle diffemorphism possesses the complex rotation number in the upper half-plane.
The complex rotation number depends analytically on a hyperbolic diffeomorphism, see Remark \ref{rem-analyticity} below.

In this article, we will need the analogous result for $\tau_F(\omega)$ instead of $\tau_f(\omega)$; recall that  $\tau_F$ is the lift of $\tau_f \colon \bbH/\bbZ \to \bbH/\bbZ $ to $\bbH$. The proof of the following result literally repeats the proof of Theorem \ref{th-XB_NG} in \cite{XB_NG}.

\begin{theorem}
\label{th-XB_NG-1}
In assumptions of Theorem \ref{th-XB_NG}, the holomorphic function $\tau_F \colon \bbH \to \bbH $ has a continuous extension $\bar \tau_F $ to $ \overline{\bbH}$. For $\omega \in \bbR$,
\begin{enumerate}
 \item If $\rot (f+{\omega})$ is irrational, then $\bar \tau_F(\omega)=\rot (F+{\omega})$.
 \item \label{it-parab}If $\rot (f+{\omega})$ is rational and $f+{\omega}$ has a parabolic periodic orbit, then again $\bar \tau_F(\omega)=\rot (F+{\omega})$.
 \item \label{it-hyp}If $f+{\omega}$ is hyperbolic with rotation number $p/q$  on an open interval  $\omega\in I\subset \bbR $, then $\bar \tau_F$ is analytic on $I$ and $\bar \tau_F(\omega)\in \bbH$ for $\omega \in I$.

 Moreover, $\bar \tau_F(\omega)$ belongs to the closed disk of radius $D_f/(4\pi q^2)$ tangent to $\bbR$ at $p/q$.
\end{enumerate}
\end{theorem}

We will also use the following estimate on $\bar \tau(F)$.
\begin{lemma}[see {\cite[Lemma 4]{XB_NG}}]
\label{lem-disc-radius}
Assume that $f$ is a hyperbolic map with rational rotation number
$\rot F =p/q$. Then $\bar \tau (F)$ belongs to the disk tangent to $\bbR$ at $p/q$ with radius
$$R=\left(2\pi q \cdot \sum_{x\in Per( f)} \frac{1}{| \log \rho_x|}\right)^{-1}$$
where  $\rho_x$ is the multiplier of $x$ as a fixed point of $f^q$.
% In addition, $R_\omega \le D_f /(4\pi q^2)$.
\end{lemma}
 This lemma gives a more accurate estimate on $\bar \tau (F)$ than the last subcase of Theorem \ref{th-XB_NG}; see \cite[Lemma 4]{XB_NG} for details.

\subsection{Bubbles}
Let $J\subset \bbR$ be an open segment. Consider a family $f_{\omega}$ of circle diffeomorphisms parametrized by $\omega \in J$. The family is called \emph{monotonic} if  $f_{\omega}(x)$ strictly increases on $\omega$ for each fixed $x$. We do not consider monotonically decreasing families because the change of variable $\omega \to -\omega$ turns decreasing families into increasing ones.  From now on, we only consider analytic families of analytic circle diffeomorphisms.
Let $F_{\omega}$ be a lift of $f_{\omega}$ to the real axis.

Put $I_{\frac pq, F_{\omega}}:=\{\omega\in J \mid \rot(F_{\omega})=\frac pq\}$ where $p/q\in \bbQ$. This interval contains several  open \emph{intervals of hyperbolicity},  where $f_{\omega}$ is hyperbolic. The complement to these intervals in $I_{\frac pq, F_{\omega}}$ is formed by several isolated points $\omega \in I_{p/q, F_{\omega}}$ such that $f_{\omega}$ has parabolic periodic orbits; this is clear because $f_{\omega}$ is strictly monotonic.

The following definition was introduced in \cite{XB_NG} for $f_{\omega}=f+\omega$.
  \begin{definition}
  The image of the segment $I_{\frac pq, F_{\omega}}$ under the map $\omega \mapsto \bar \tau(F_\omega)$ is called the \emph{$\frac pq$-bubble} of the family $f_{\omega}$ and denoted by $B_{p/q, F_{\omega}}$.
 \end{definition}
 Lemmas \ref{lem-parab}, \ref{lem-hyp} below imply that the map  $\omega \mapsto \bar \tau(F_\omega)$ is continuous on the segment $I_{\frac pq, F_{\omega}}$, thus the $p/q$-bubble of a monotonic family is a continuous curve.
 Due to Remark \ref{rem-analyticity}, $\bar\tau(f_{\omega})$ is analytic on $\omega$. This remark and Theorem \ref{th-XB_NG-1} (case \ref{it-parab}) show that the images of the intervals of hyperbolicity are analytic curves, and the images of their endpoints are at $p/q$. So the $p/q$-bubble is a union of several analytic curves in the upper half-plane ``growing'' from $p/q$ and the point $p/q$ itself. It is possible that $I_{\frac pq, F_{\omega}}$ is just one point; then $B_{\frac pq, F_{\omega}}$ is also a point, $B_{\frac pq, F_{\omega}}=\{p/q\}$.

 Each monotonic family of circle diffeomorphisms $f_{\omega}$ gives rise to the ``fractal-like'' set $\bigcup_{p/q\in \bbQ} B_{p/q, F_{\omega}}$ (bubbles) in $\bbH$, containing countably many analytic curves ``growing'' from rational points.
These curves may intersect and self-intersect as shown in  \cite{NG-inters}; some pictures of bubbles are also presented there.
The aim of this article is to prove that the bubbles of monotonic families are approximately self-similar near rational points, and to describe the ``limit shapes'' of bubbles.

\subsection{Self-similarity of bubbles near zero}

Let $f_{\eps}$, $\eps \in J$, be an analytic monotonic family of circle diffeomorphisms; here $J\subset \bbR$ is an open interval. Let $F_{\eps}$ be a lift of $f_{\eps}$. We consider the bubbles near zero, i.e. $B_{r, F_{\eps}}$ as $r\to 0, r\in \bbQ^+$ (for $r<0$, the consideration is analogous). The corresponding intervals $I_{r, F_{\eps}}$ accumulate to the right endpoint of $I_{0,F_{\eps}}$. We assume that this right endpoint is $\eps=0$ and write $f:=f_0$, $F:=F_0$.

Since $0$ is the right endpoint of $I_{0,F_{\eps}}$, we have ${\rot F = 0}$, $F$ has only parabolic fixed points, and we must have $F(x)\ge x$ for all $x\in \bbR$, so that the parabolic fixed points of $F$ disappear as $\eps$ increases.

We impose additional genericity assumptions on $f_{\eps}$. We assume that  $f$ has the only parabolic fixed point, and shift it to $0$, so that $F(0)=0$, $f'(0)=1$. We assume that this fixed point has  multiplicity 2, i.e. $f''(0)\neq 0$. We also assume $\frac{\partial f}{\partial \eps} (0)>0$.

In Sec. \ref{sec-Fatou}, we present the construction of a circle diffeomorhpism ${\mathbf K}$; in an appropriate chart, the family ${\mathbf K}+c$ becomes  the family of Lavaurs maps for the parabolic fixed point $0$ of $F$, and $c $ is the Lavaurs phase (see Definition \ref{def-Lavaurs}).
It turns out that self-similarity patterns of bubbles of $f_{\eps}, \eps>0$, are related to the bubbles of ${\mathbf K}$. The map~${\mathbf K}$ (``transition map'') for $C^2$ circle diffeomorphisms was introduced in \cite{Young} as a modulus of $C^1$ classification; its  role in the bifurcations of parabolic points of circle diffeomorphisms was studied in \cite[Theorem 6]{Afr_Liu_Young} and  \cite{Afr_Young}.

Put $R(z):=-1/z$. Take any rational number $p/q\in \bbQ, 0\le p/q \le 1$. Denote $a_n = -\frac{1}{(p/q)-n}$; then  $a_n \to 0$ and $R(a_n)=(p/q)-n$.

\begin{theorem}[Limit shapes of bubbles-1]
\label{th-main-1}
Let a monotonic analytic family of circle diffeomorphisms $f_{\eps}$ be as above.
In the above notation, the set of limit points of the curves $R(B_{a_n, F_{\eps}}) \mod 1 \subset \bbC/\bbZ$ as $n \to \infty$ includes the $p/q$-bubble of the family  ${\mathbf K} + c$.
\end{theorem}
With an additional requirement on ${\mathbf K}$, we get the following stronger result.

\begin{theorem}[Limit shapes of bubbles-2]
\label{th-main}
In assumptions of Theorem \ref{th-main-1}, suppose that for all $c$, whenever $\rot ({\mathbf K}+c)=p/q$, the diffeomorphism ${\mathbf K}+c$ has at most one parabolic cycle.

Then the curves $R(B_{a_n, F_{\eps}}) \mod 1  \subset \bbC/\bbZ$ (with some parametrizations) tend uniformly to the $p/q$-bubble of the family  ${\mathbf K} + c$ as $n\to \infty$.
\end{theorem}
\begin{figure}[ht]
 \includegraphics[width=0.7\textwidth]{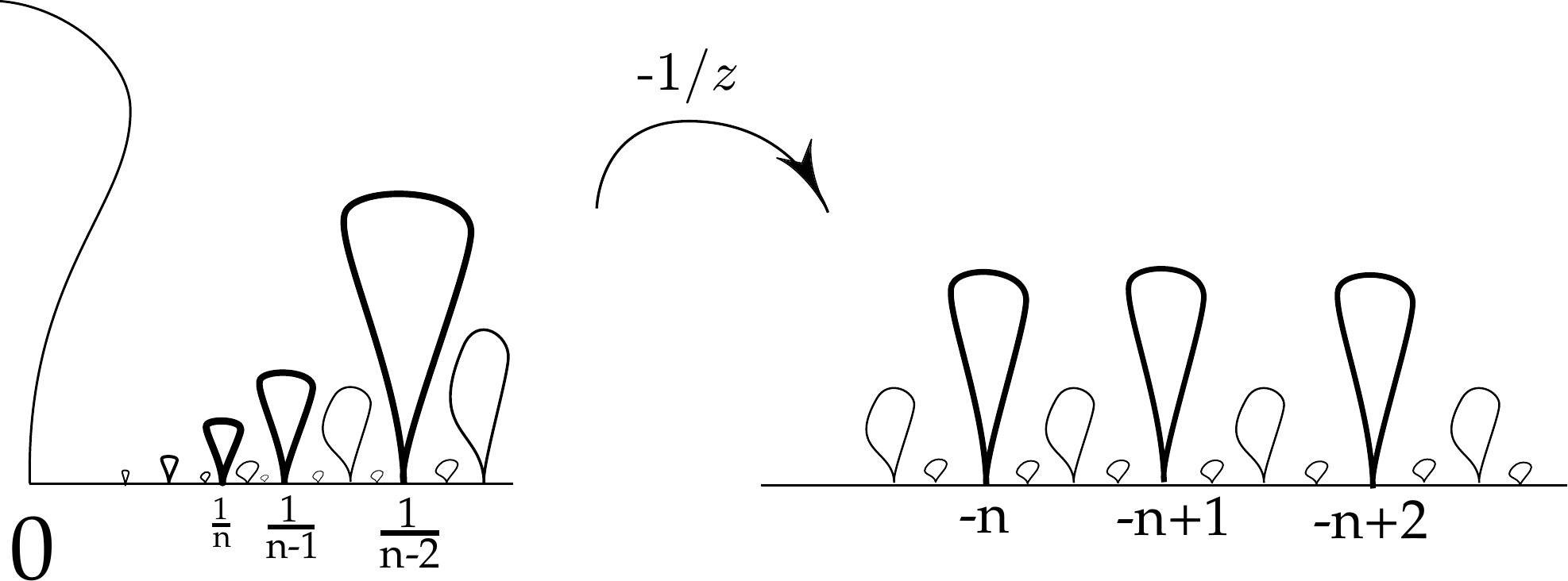}
 \caption{Self-similarity of bubbles. The $a_n$-bubbles with $a_n=1/n, p/q=0$, are shown in thick.}
\end{figure}

The weaker version of this theorem holds true even in assumptions of Theorem \ref{th-main-1}, see Theorem \ref{th-main-1.5} below; we postpone its statement till Sec. \ref{sec-proofs} because it is more technical.

\begin{remark}
\label{rem-genericity}
The additional requirement in Theorem \ref{th-main} (that the maps $\mathbf K+c$ with rotation number $p/q$  do not have two parabolic cycles simultaneously) is open and dense, see the sketch of the proof in Sec. \ref{sec-generic}. This condition cannot be removed. The reason for this is that the complex rotation number is not continuous at diffeomorphisms with two parabolic cycles, see Remark \ref{rem-discont}. The requirement is used in Lemma \ref{lem-tau-contin}.
\end{remark}

Theorem \ref{th-main} implies the following self-similarity.
\begin{theorem}[Self-similarity of bubbles-1]
\label{th-ssim}
 In assumptions of Theorem \ref{th-main}, the set $\bigcup_n B_{a_n, F_{\eps}}$ is approximately self-similar near the point $0$, with the self-similarity given by $z \mapsto \frac{z}{1+z}$.

 Formally, the distance in $C$ metrics between the curves $R(B_{a_n, F_{\eps}})$ and $R(B_{a_{n+1}, F_{\eps}})+1$  (with some parametrizations) tends to zero as $n\to \infty$.
\end{theorem}
Note that $R$ conjugates the shift $z\mapsto z-1$ to   $z \mapsto \frac{z}{z+1}$, which motivates the informal statement that the self-similarity of $\bigcup_n B_{a_n, F_{\eps}}$ is given by $z\mapsto \frac{z}{z+1} $.

\begin{proof}
Theorem~\ref{th-main} implies that the distance between $R(B_{a_n, F_{\eps}})$ and $R(B_{a_{n+1}, F_{\eps}})$ in $\bbC/\bbZ$ tends to zero. Recall that the map $R$ takes the sequence $\{a_n\}$ to $\{(p/q)-n\}$, so it takes the bubble $B_{a_n, F_{\eps}}$ to a continuous curve with both endpoints at $(p/q)-n$. Since  $R(B_{a_n, F_{\eps}})$ and $R(B_{a_{n+1}, F_{\eps}})$  start at points $(p/q)-n$ and $(p/q)-n-1$ respectively, and  are close in $\bbC/\bbZ$, the conclusion follows.

\end{proof}

Another version of a self-similarity result is the following.
\begin{theorem}[Self-similarity of bubbles-2]
\label{th-ssim-2}
 In assumptions of Theorem \ref{th-main-1}, suppose that ${\mathbf K}+c$ has at most one parabolic cycle for each $c$. Then the whole set of bubbles $\bigcup_{a,b} B_{a/b, F_{\eps}}$ is approximately self-similar near $0$.

 Formally, we have the following convergence of the countable unions of analytic curves:
 \begin{equation}
  \label{eq-allbub}
  R\left(\bigcup_{\frac ab \in \bbQ, \frac ab\in  [\frac 1n, \frac 1{n+1}]} B_{a/b, F_{\eps}}\right) +n \to \left(\bigcup_{\frac{p}{q} \in \bbQ, \frac pq\in  [0, 1]} B_{p/q, \mathbf K +c}\right)
 \end{equation}
as $n\to \infty$, and the convergence is uniform.
\end{theorem}
\begin{proof}
For a small $\delta>0$, we should prove that for $n$ large, the above union is $\delta$-close to the bubbles of ${\mathbf K}+c$.

Clearly, any rational number $a/b>0$ sufficiently close to $0$ appears in a sequence $\{a_n\}$ for some appropriate $p/q$, namely $p/q = \{-b/a\}$. If in the union  \eqref{eq-allbub} we only take $a/b$ that correspond to  finitely many $q$, the statement follows from Theorem \ref{th-main}. We will see that for $a/b$ corresponding to large $q$, the statement holds authomatically, which will finish the proof.

Formally, let $a/b$ be small and appear in a sequence $\{\frac 1{(p/q)-n)}\}$: $a = q, b = p-qn$.  Then due to Theorem \ref{th-XB_NG-1} case 3,  the bubble $B_{a/b, F_{\eps}}$ is within the $C/b^2$-neighborhood of $a/b$. The direct computation shows that $R(B_{a/b, F_{\eps}})$ is within the $C'/q^2$-neighborhood of $p/q$. Here $C, C'$ are constants that depend on the family $f_{\eps}$ only. Thus for sufficiently large $q$, $B_{a/b, F_{\eps}}$ belongs to the $\delta$-neighborhood  of $B_{p/q, {\mathbf K}+c}$.

\end{proof}

\begin{remark}
Above we mentioned (Remark \ref{rem-genericity}) that for any $p/q$, the additional requirement on ${\mathbf K}$ in Theorem \ref{th-main} corresponds to an open and dense set of $f$. Hence  the additional requirement on ${\mathbf K}$ in Theorem \ref{th-ssim-2} corresponds to a residual set of $f$.
\end{remark}

\subsection{Self-similarity of bubbles near any rational point}
Analogous self-similarity results hold near any rational point of the real axis. We do not repeat remarks from the previous section; all of them apply for the theorems below as well.

Consider an analytic monotonic family of circle diffeomorphisms $f_{\eps}$, $\eps \in J$, and let $F_{\eps}$ be a lift of $f_{\eps}$. Fix a rational number $k/l$.  We consider the bubbles $B_{s, F_{\eps}}$ where $s \to k/l$, $s>k/l$. Again, we assume that the right endpoint of $I_{k/l, F_{\eps}}$ is zero,  and write $f:=f_0$, $F:=F_0$.

Then we  have $\rot F = k/l$, $f$ has only parabolic cycles, and $F^l(x)-k\ge x$ for all $x\in \bbR$, so that the parabolic cycles of $f$ disappear as $\eps$ increases.

We impose additional genericity assumptions on $f_{\eps}$: we assume that  $f$ has the only parabolic cycle, namely the orbit of $0$; we also  assume that this cycle has  multiplicity 2, i.e. $(f^l)'(0)=1$, $(f^l)''(0)\neq 0$, and $\frac{\partial f^l_{\eps}}{\partial \eps}(0)>0$.

In Sec. \ref{sec-Fatou}, we present the construction of a circle diffeomorhpism $\tilde {\mathbf K}$ and explain its relation to Lavaurs maps of $f^l$.

Take $r,s\in\bbZ $ such that $0<r<l$ and $kr+ls=1$. Let $\tilde R(z)=\frac{rz+s}{-lz+k}$ (so that $\tilde R \in SL(2,\bbZ)$).
Given $p/q\in\bbQ$, $0\le p/q \le 1$, we put $\tilde a_n :=\tilde R^{-1} (p/q-n)$; then  $\tilde a_n \to k/l$.

% Let $a:=f^{-r}(0)$ and $b:=f^r(0)$ be the closest points of the orbit of $0$ to $0$, $a<0<b$.
%  Let $\Psi^{\pm} \colon U^{\pm}\to \bbC$ be Fatou coordinates of $F^l-k$ at $0$. Then $\Psi^{\pm} (F^l(z)-k) = \Psi^{\pm}(z)+1$, $U^-\subset \bbC$ contains $(a,0)$ and $U^+\subset \bbC$ contains $(0,b)$. Moreover, $\Psi^{+}((a, 0)) =\bbR$ and  $\Psi^{-}((0,b)) =\bbR$.
%  Note that the map $z \mapsto \Psi^-(F^{-r}(\Psi^+)^{-1}(z))$ is well-defined on $\bbR$ because both  $\Psi^-$ and $\Psi^+(F^r(z)) $ take $(a,0)$ to $\bbR$. Moreover, this map conjugates the shift by $1$ to itself, thus defines a circle diffeomorphism. We will denote it by $\tilde {\mathbf K}$, $\tilde {\mathbf K}\colon \bbR/\bbZ \to \bbR/\bbZ$.
\begin{theorem}[Limit shapes of bubbles-1]
\label{th-main-1-kl}

In the above notation, the set of limit points of the curves $\tilde R(B_{\tilde a_n, F_{\eps}}) \mod 1 \subset \bbC/\bbZ$ as $n \to \infty$ includes the $p/q$-bubble of the family $\tilde {\mathbf K} + c$.
\end{theorem}
\begin{theorem}[Limit shapes of bubbles-2]
\label{th-main-kl}
In assumptions of Theorem \ref{th-main-1-kl}, suppose that for all $c$, whenever $\rot (\tilde {\mathbf K}+c)=p/q$, the diffeomorphism $\tilde {\mathbf K}+c$ has at most one parabolic cycle.

Then the curves $\tilde R(B_{\tilde a_n, F_{\eps}}) \mod 1$ (with some parametrizations) tend uniformly to the $p/q$-bubble of  the family $\tilde {\mathbf K} + c$ as $n\to \infty$.
\end{theorem}
The proofs of the following results literally repeat the proofs of Theorems \ref{th-ssim} and \ref{th-ssim-2}.

\begin{theorem}[Self-similarity of bubbles-1]
\label{th-ssim-kl}
 In assumptions of Theorem \ref{th-main-kl}, the set $\bigcup_n B_{\tilde a_n, F_{\eps}}$ is approximately self-similar near the point $k/l$, with the self-similarity given by $\tilde R^{-1}( \tilde R -1)$.

 Formally, the distance in $C$ metrics between the curves $\tilde R(B_{\tilde a_n, F_{\eps}})$ and $\tilde R(B_{a_{n+1}, F_{\eps}})+1$  (with some parametrizations) tends to zero as $n\to \infty$.
\end{theorem}
\begin{theorem}[Self-similarity of bubbles-2]
\label{th-ssim-1-kl}
 In assumptions of Theorem \ref{th-main-1-kl}, suppose that $\tilde {\mathbf K}+c$ for all $c$ has at most one parabolic cycle. Then the whole set of bubbles $\bigcup_{a,b} B_{a/b, F_{\eps}}$ is approximately self-similar near $p/q$.

 Formally, the countable union of analytic curves $$\tilde R\left(\bigcup_{\frac ab \in \bbQ, \frac ab\in  \tilde R^{-1}[-n, -n+1]} B_{a/b, F_{\eps}}\right) +n$$ tends uniformly to the bubbles of $\tilde {\mathbf K}+c$ as $n\to \infty$.
\end{theorem}

\subsection{The definition of ${\mathbf K}$, $\tilde {\mathbf K}$}
\label{sec-Fatou}
\subsubsection{Fatou coordinates}
We define Fatou coordinates for a parabolic diffeomorphism $F\colon (\bbC, 0)\to (\bbC, 0)$ having a fixed point of multiplicity 2 at $0$ and such that $F(x)\ge x$. See \cite{Shi} for the proofs of the statements below, and \cite{Dou} for sketches.

\begin{figure}
\begin{center}
 \includegraphics[width=0.7\textwidth]{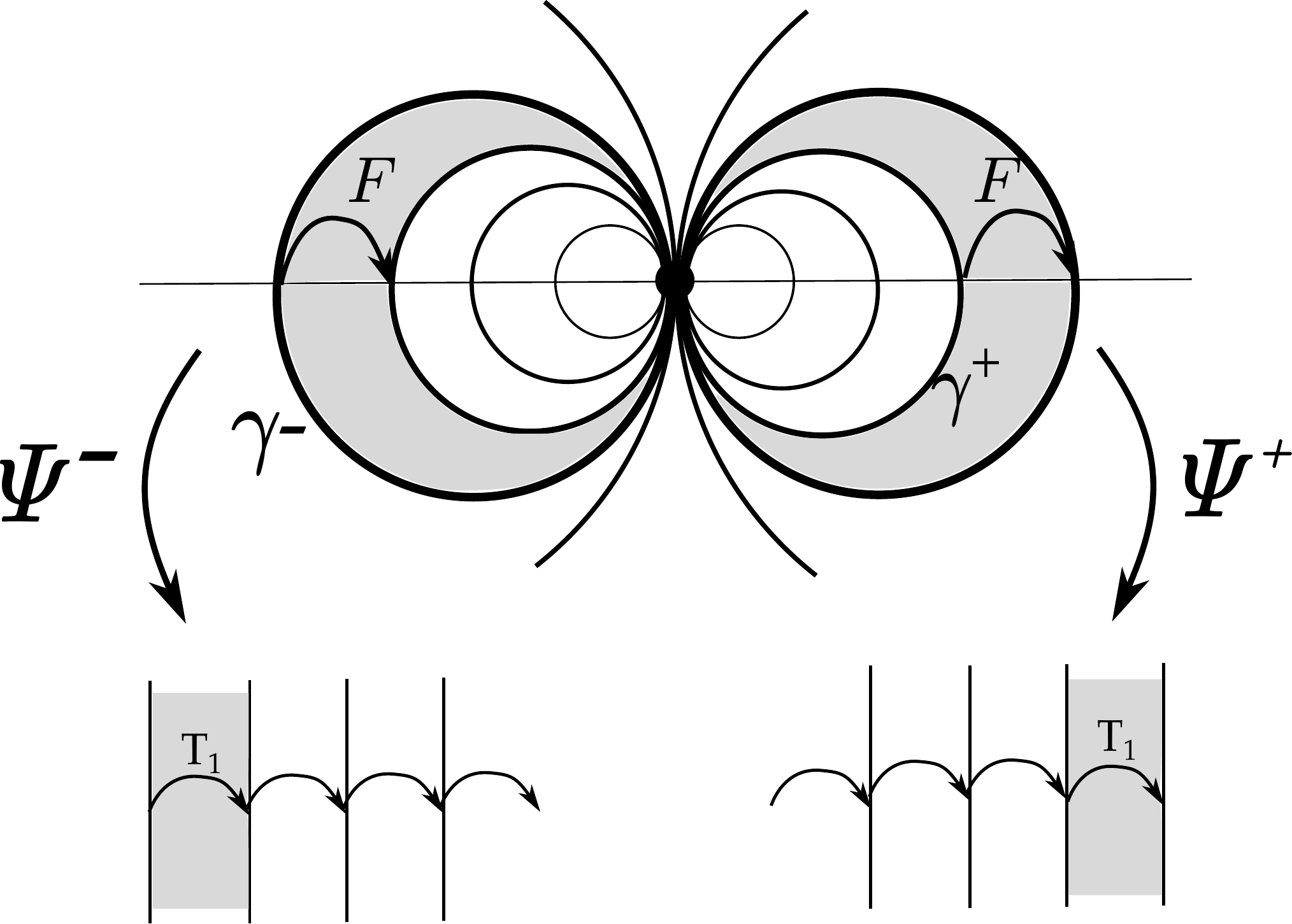}
 \end{center}
 \caption{Fatou coordinates}\label{fig-Fatou}
\end{figure}

Let $\gamma^{-}$ be a small circle centered on $\bbR^-$ and passing through $0$. Then  $F(\gamma^{-})$  does not intersect $\gamma^{-}$ (except at $0$).
Let $\Pi$ be the domain between $\gamma^{-}$ and $F(\gamma^{-})$ with $0$ removed. Its quotient space by the action of $F$ is a cylinder (Ecalle cylinder) biholomorphically equivalent to $\bbC/\bbZ$. Let $\Psi^{-}\colon \Pi \to \bbC$ be the lift of this biholomorphism to $\Pi$. Then $\Psi^{-}$ conjugates $F$ to the shift $T_1 \colon z\mapsto z+1$ and  extends, via iterates of $F$, to some domain $U^-$. The map $\Psi^-$ is called the attracting  Fatou coordinate of $F$.

Similarly we may construct the repelling  Fatou coordinate $\Psi^{+}$ on some domain $U^+$; the corresponding circle $\gamma^+$ is in the right half-plane. The union $U^+\cup U^-$ covers the neighborhood of zero.

Both maps $\Psi^{\pm}$ are well-defined up to a shift. We will normalize them by requiring $\Psi^-(x)=0, \Psi^+(y)=0,$ for some points $x, y\in\bbR, x<0<y$. If $F$ preserves the real axis, then all constructions are symmetric with respect to $z\mapsto \bar z$, so $\Psi^{\pm}$ preserve the real axis.

\begin{definition}
\label{def-Lavaurs}
\emph{Lavaurs maps} for $F$ are the maps $L_c:=(
\Psi^{+})^{-1} T_c\Psi^-$, $L_c\colon U^-\to\bbC$, for all $c\in\bbC$, where $T_c(z)=z+c$. The number $c$ is called the \emph{Lavaurs phase}.
\end{definition}

\subsubsection{The definition of $\mathbf K$}
\label{ssec-K}
In assumptions of Theorem \ref{th-main-1}, the chart $\Psi^{-}$ extends to $(-1,0)$ via iterates of $F$, and $\Psi^{+}$ extends to $(0,1)$. We normalize them by $\Psi^-(x)=0$ and $\Psi^{+}(x+1)=0$ for some $x \in (-1,0)$.

On the circle, the domains of definition of $\Psi^{\pm}$ coincide. This enables us to consider the transition map between Fatou coordinates. Formally, we put
$${\mathbf K}:=\Psi^- ( (\Psi^+)^{-1}-1).$$ This map is defined on the whole real axis and commutes with the shift by $1$, hence defines an analytic circle diffeomorphism. This is the map we need for Theorems \ref{th-main-1} and \ref{th-main}.

Another viewpoint on the same map is the following. Between the fundamental domains $[x, F(x)]$ and $[x+1, F(x+1)]$, we have two natural maps: one is the shift by $1$, the other one is a Lavaurs map. Now, $\mathbf K$ is their composition, in the appropriate chart. This is formalized as follows.
\begin{remark}
 \label{rem-Lav-chart}
  In the chart $\Psi^+(z+1) = \Psi^+ T_1$ on $(-1, 0)$, the map ${\mathbf K}+c$ turns into the Lavaurs map $L_c-1$:
  $$T_{-1}(\Psi^+)^{-1}T_c \Psi^- ((\Psi^+)^{-1}-1)\Psi^+ T_1   = T_{-1}(\Psi^+)^{-1} T_c\Psi^- =  L_c -1.$$
 \end{remark}

 \subsubsection{The definition of $\tilde{ \mathbf K}$}

In assumptions of Theorem \ref{th-main-1-kl}, the map $f^l$ has a parabolic fixed point of multiplicity $2$ at zero, and its orbit under $f$ consists of $l$ points.
Let $a:=f^{-r}(0)$ and $b:=f^r(0)$ be the closest returns of this orbit to $0$, $-1<a<0<b<1$, $0<r<l$. Both $a,b$ are fixed points for $f^l$.
Consider Fatou coordinates $\Psi^{\pm}$ of $0$ as a fixed point of $F^l-k$.
Then $\Psi^{-}$ extends to $(a,0)$ via iterates of $F^l-k$, and $\Psi^{+}$ extends to $(0,b)$.
We choose the normlization of $\Psi^{\pm}$ so that $\Psi^-(x)=0$ and $\Psi^+(f^r(x))=0$ for some $x\in (a,0)$.

Now, there are two coordinates defined on $(a,0)$, namely $\Psi^-$ and the Fatou coordinate of $a$ as the fixed point of $f^l$. The latter is the image of $\Psi^+$ under $f^{-r}$. We consider the transition map
$$\tilde {\mathbf K}:=\Psi^- (f^{-r} (\Psi^+)^{-1}).$$
It is well-defined on $\bbR$ and commutes with the shift by $1$, thus defines an analytic circle diffeomorphism; this is the map we need for Theorems \ref{th-main-1-kl} and \ref{th-main-kl}.

Another viewpoint on the same map is the following. Between the fundamental domains $[x, f^l(x)]\subset (a, 0)$ and $f^r ([x, f^l(x)])\subset (0,b)$, there are two natural maps: namely, $f^l$ and the Lavaurs map of $0$ for $F^l$. The map $\tilde {\mathcal K}$ in the appropriate chart is their difference. Formally, we have the following statement.
 \begin{remark}
 \label{rem-Lav-chart-kl}
  In the chart $\Psi^+$ on $(0,b)$, the map $\tilde {\mathbf K}+c$ turns into $$(\Psi^+)^{-1}T_c \Psi^- f^{-r}(\Psi^+)^{-1}\Psi^+  = (\Psi^+)^{-1} T_c\Psi^- f^{-r}=  L_c f^{-r}$$
  where $L_c$ is a Lavaurs  map for $F^l$.
 \end{remark}
The above construction is a generalization of the construction from Sec. \ref{ssec-K}, with $l=1$, $r=0$.

\begin{remark}
 The number $r$ introduced above coincides with the number $r$ from Theorem \ref{th-main-1-kl}. Indeed, the (periodic) orbit of $0$ under $f$ is ordered in the same way as any  orbit of the translation $T_{k/l} (x)=x+k/l$, because $\rot f = k/l$. For the translation, the closest point of the orbit of $0$ to $0$ is $1/l$. So the number $r$ above satisfies  $(T_{k/l})^{r}(x) = x+1/l$. This implies $kr \equiv 1 \mod l$, and there exists the only integer $r, 0<r<l,$ with this property. For this $r$, we have $kr -1 \equiv 0 \pmod l$, i.e. there exists an integer $s$ such that $kr+ls=1$. So this is the number $r$ from Theorem \ref{th-main-1-kl}.
\end{remark}
\subsection{Genericity of the additional restriction on $\textbf{K}$ in Theorem \ref{th-main} }
\label{sec-generic}
 Let $\mathbf K_f := \mathbf K$ be the transition map that corresponds to the parabolic diffeomorphism $f$. Recall that the additional assumption on $\textbf{K}$ in Theorem \ref{th-main} was as follows:
 \begin{equation}
\begin{aligned}
& \text{ ``For any $c$, whenever $\rot (\mathbf K_f+c)=p/q$, the map $\mathbf K_f+c$ has}\\
& \text{  at most one parabolic orbit''.}
\end{aligned}
\label{eq-cond}
\end{equation}
We will sketch the proof of the fact that this condition holds for a generic $f$. The same result holds true for $\tilde {\mathbf K}$.
 \begin{lemma}
For any fixed $p/q$, the set of parabolic diffeomorphisms $f$ that correspond to $\mathbf K_f$ satisfying
\eqref{eq-cond}  is open and dense in the metric of uniform convergence.
 \end{lemma}
\begin{proof}[Sketch of the proof]
Clearly, \eqref{eq-cond} defines an open and dense set in the space of all circle diffeomorphisms. So we will focus on the map $f \mapsto \mathbf K_f$, to see that the set of corresponding $f$ is open and dense as well.

 It is well-known (see \cite[Proposition 2.5.2 (iii)]{Shi}) that Fatou coordinates depend continuously on a (parabolic) map $F$. Thus the mapping $f \mapsto \mathbf K_f$ is continuous. Therefore the set under consideration is open.

Now it suffices to prove that we may perturb any initial circle diffeomorphism $f$ to achieve an arbitrary perturbation of  $\mathbf K_f$ (i.e. that the mapping $f\mapsto \mathbf K_f$ is open). This will imply the statement.

The idea of the proof is the following: we fix the diffeomorphism $f$, cut the circle where $f$ acts, and glue again, with the gluing close to identical. This produces a new circle diffeomorphism on a ``new'' circle. In an appropriate chart on the ``new'' circle, the new diffeomorphism is close to $f$. The transition map $\mathbf K_f$ between Fatou coordinates  changes in a controllable way which implies the statement.

In more detail, fix $F$ and a point $x, -1<x<0$. Consider the interval $I=[x, F(x+1)]$ with the map  $F|_{[x,x+1]}$ on it. Here we ``cut'' the circle where $f$ acts: namely, $F|_{[x,x+1]}$ induces $f$ on $I/T_1\sim \bbR/\bbZ$. Choose another gluing: take any analytic map $h\approx T_{-1}$ that takes a neighborhood of $[x+1, F(x+1)]$ to a neighborhood of $[x, F(x)]$ and commutes with $F$; consider the quotient $I/h$ (``new'' circle). This quotient is a one-dimensional real-analytic manifold homeomorphic to a circle, thus it \emph{is} a circle: there exists a real analytic map $H\colon \bbR/\bbZ \to I/h$. Its lift $\tilde H \colon \bbR\to I$ conjugates $h$ to the shift by $(-1)$, $H(x-1) = h(H(x))$. We will see that $H$ is close to identity.

 Thus $\tilde f= H^{-1}fH$ is an analytic circle diffeomorphism close to $f$. Its Fatou coordinates are $\Psi^{\pm}H$, and the corresponding transition map is $\mathbf K_{\tilde f} = \Psi^-H (H^{-1}\Psi^+)^{-1}-1) = \Psi^-(h(H H^{-1}(\Psi^+)^{-1})) =\Psi^- h(\Psi^+)^{-1} $. Since $h$ is an arbitrary map close to the shift $T_{-1}$ that commutes with $F$, we may achieve any perturbation of $\mathbf K_f$.

 It remains to prove that  $H$ is close to identity. The proof relies on Ahlfors-Bers theorem, and we only sketch it here. Namely, we consider a neighborhood $V$ of $I$; $V/h$ is an annulus. It is easy to find a smooth map $R_1$ close to $id$ in $C^2$, such that $R_1$ takes $V/h$ to a standard annulus $A$ and commutes with $z \to \bar z$. Now, $R_1$ induces a conformal structure on $A$. It is close to the standard conformal structure, hence  due to Ahlfors-Bers theorem, there exists a quasiconformal map $R_2\approx id$ that uniformizes this conformal structure. The uniqueness of the (normalized) uniformization implies that $R_2$ commutes  with $z \mapsto \bar z$. Finally, we take $H := R_1^{-1} R_2^{-1}$, which is close to identity because both $R_1, R_2$ are close to identity, and $H$ preserves the real axis because it commutes with $z \mapsto \bar z$.
This completes the proof.
\end{proof}

\subsection{Lavaurs theorem}
The following theorem was proved by Lavaurs \cite{Lav}, see also \cite[Proposition 3.2.2]{Shi}.
% \begin{theorem}
% \label{th-Lav}
% Let  $F_{\eps}$ be an analytic family of the form $F_{\eps}(z)= z+z^2+\eps+o(z^2, \eps)$\fixme{verify; prenormalize?}.
% Suppose that $\eps_k\to 0$ and $n_k\to \infty $ satisfy
% \begin{equation}
% \label{eq-phase}
%  \lim_{k\to \infty}n_k-\frac{\pi}{\sqrt{\eps_k}}=c.
% \end{equation}
% Then $F_{\eps_k}^{n_k}\to L_c$ uniformly on compacts in $U^-$.
% \end{theorem}
% It is \emph{aposteriori} clear that the condition \eqref{eq-phase} is equivalent to the convergence of $F_{\eps_k}^{n_k}(x)$ to $L_c(x)$.
% Indeed, if \eqref{eq-phase} holds, Theorem \ref{th-Lav} implies that $F_{\eps_k}^{n(\eps_k)}(x)$ converge to $L_c(x)$. On the other hand, suppose that   $F_{\eps_k}^{n(\eps_k)}(x)$ converge to $L_c(x)$. Let $d$ be a partial limit of the sequence $\frac{\pi}{\sqrt{\eps_k}}-n_k$; again by Theorem \ref{th-Lav}, the sequence $F_{\eps_k}^{n(\eps_k)}$ has a subsequence that converges to $L_d$, thus $F_{\eps_k}^{n(\eps_k)}(x)\to L_d(x)$. But we know that $F_{\eps_k}^{n(\eps_k)}(x)\to L_c(x)$, so $d=c$ and  $\frac{\pi}{\sqrt{\eps_k}}-n_k\to c$.
%
% This implies the following version of Lavaurs theorem.
\begin{theorem}
\label{th-Lav-1}
Let  $F_{\eps}$, $\eps\in\bbR$, be an analytic family of analytic maps in a neighborhood of zero, satisfying $F_0(z)= z+az^2+\dots$, $a\neq 0$. Let $L_c$ be Lavaurs maps for $F_0$.

Suppose that $F_{\eps}$ has two complex hyperbolic fixed points near zero for $\eps>0$ small. Assume that both multipliers $\mu_{1,2}$ of these points satisfy $\arg |\frac {1}{2\pi i}\log \mu_{1,2}|<\pi/4 $.
Suppose that $\eps_k\to 0$ and $n_k\to \infty $ satisfy
\begin{equation}
\label{eq-Lav}
 \lim_{k\to \infty} F^{n_k}_{\eps_k}(x) = L_c(x).
\end{equation}
Then $F_{\eps_k}^{n_k}\to L_c$ uniformly on compact sets in $U^-$.
\end{theorem}
The above assumption on multipliers for our family $F_{\eps}$ follows from our genericity assumption $\frac{\partial F}{\partial \eps}(0)>0$.
See also \cite[Propositions 13.1 and  18.2]{Dou} for the particular case $F_{\eps}(z)=z+z^2+\eps$ which does not essentially differ from the general case.

\section{Proof of the main theorem}
\subsection{Renormalizations and complex rotation numbers}

It is well-known that the renormalization of a circle diffeomorphism with rotation number $\rho, 0<\rho<1$, has the rotation number $-1/\rho$.
It turns out that the same result holds for complex rotation numbers. First, we recall the definition of renormalization.

 Take a fundamental domain $[x,g(x)]$ of an analytic circle diffeomorphism $g$ with no fixed points. The first return map to this domain under iterates of $g$ commutes with $g$, hence it descends to the quotient $[x,g(x)] / g$. Such self-map $\mathcal R g$ of  $[x,g(x)] / g$ is called  the \emph{renormalization} of $g$.

 Clearly, the quotient is a one-dimensional real-analytic manifold analytically equivalent to a circle. On this circle, $\mathcal R g$ is an analytic circle diffeomorphism. It is well-defined up to an analytic coordinate change.

 There is no canonical choice of an analytic chart on the circle $[x,g(x)] / g$. However the complex rotation number does not depend on the choice of the analytic chart, due to the following lemma (for the proof, see \cite[Lemma 8]{NG-inters} or Remark \ref{rem-depend-on-g+omega} below).
\begin{lemma}
\label{lem-chart}
 Complex rotation number $\bar \tau|_{\bbR}$ is invariant under analytic conjugacies:
for two analytically conjugate circle diffeomorphisms $f_1,f_2$, we have $\bar \tau(f_1) = \bar \tau(f_2)$.
\end{lemma}
So $ \bar \tau(\mathcal R g)$ is well-defined. The following lemma relates $\bar \tau(\mathcal R g)$ to $\bar \tau(g)$. Recall that $R(z)=-1/z$.

\begin{lemma}[Complex rotation numbers under renormalizations]
\label{lem-renorm}
Let $g$ be an analytic circle diffeomorphism with no fixed points, let $G$ be its lift to the real line with $0<\rot G<1$. Then
\begin{equation}
\label{eq-Rg-g}
\bar \tau (\mathcal R  g) \equiv -\frac{1}{ \bar \tau(G)} \equiv R( \bar \tau(G))  \pmod 1.
\end{equation}
\end{lemma}
The proof is postponed till Sec. \ref{sec-renorm-proof}.
\subsection{Renormalization at a rational point}
If the rotation number of a circle diffeomorphism $g$ is close to $k/l$, it is reasonable to consider the following $k/l$-renormalization of $g$.

Suppose that $\rot g>k/l$ is sufficiently close to $k/l$.  We consider the first-return map under $g$ to the segment $[x, g^l(x)]$. This map descends to the well-defined map $\mathcal R ^{k/l}g$ on $[x, g^l(x)]/g^l$. Then $\mathcal R^{k/l}g$ is called the $k/l$-\emph{renormalization} of $g$.
It is induced by some powers of $g$; the following lemma gives an explicit form of the first-return map.
\begin{lemma}
 \label{lem-renorm-kl}
Let $k/l\in \bbQ$ be an irreducible fraction, and let $r$, $0<r<l$, be such that $kr \equiv 1 \mod l$. Let $g$ be a circle diffeomorphism with $\rot g>k/l$ sufficiently close to $k/l$.

Then for some $n$, the first-return map on $[x,g^l(x)]$ under $g$ has the form $g^{nl-r}$ on $[x,a]$ and $g^{(n-1)l-r}$ on $[a, g^l(x)]$, where $a = g^{l(1-n)+r}(x)$.
 \end{lemma}
 The proof is postponed till Sec. \ref{sec-proof-lem-kl}.

The following lemma is an analogue of Lemma \ref{lem-renorm}.
\begin{lemma}
\label{lem-renorm-1}
Let $g$ be a circle diffeomorphism with $\rot g >k/l$ sufficiently close to $k/l$. Let $G$ be its lift to the real line such that $\rot (G^l-k)\in (0,1)$. Let $r,s$ be integer numbers such that $rk+ s l =1$.  Then
 $$\bar \tau (\mathcal R^{k/l}g) \equiv \frac{r\bar \tau (G) +s}{-l\bar \tau (G)+k} \pmod 1.$$
\end{lemma}
Note that the above expression for $\bar \tau (\mathcal R^{k/l}g)$ equals $\tilde R (\bar \tau (G))$ where $\tilde R $ is defined as in Theorem \ref{th-main-1-kl}.
The proof is postponed till Sec. \ref{sec-renorm-proof}.

\subsection{Parabolic renormalization: through the eggbeater}
\label{sec-egg}
In this section, we will see that the renormalizations of the maps $f_{\eps}$ tend to the family of Lavaurs maps $L_c$. This is a well-known corollary of the Lavaurs theorem, however we provide a proof due to the lack of a suitable reference. We will prove this fact in  analytic charts close to $\Psi^+ (z+1)$, and the maps $L_c$ will turn into the family $\mathbf K+c$.

This induces a reparametrization of $(0,\eps)$, more or less by this $c$. We study this parametrization in Lemmas \ref{lem-d} and \ref{lem-d-monot} of this section. Now let us pass to more details.

In assumptions of Theorem \ref{th-main-1}, recall that $\Psi^{\pm}$ are Fatou coordinates of $F$ at zero, normalized by $\Psi^-(x)=0$ and $\Psi^+(x+1)=0$, where $-1<x<0$. For each small $\eps>0$, consider renormalizations $\mathcal Rf_{\eps}$, of our family $f_{\eps}$, associated with fundamental domains $[x,F_{\eps}(x)]$. Each map $\mathcal Rf_{\eps}$ acts on its own circle $[x,F_{\eps}(x)]/F_{\eps}$.

% Note that the map $z \mapsto \Psi^+_{\eps}(z+1) \mod 1$ is well-defined on the quotient and takes $[x,F_{\eps}(x)]/F_{\eps}$  to $\bbR/\bbZ$.

% We normalize Fatou coordinates $\Psi^-_{\eps}$ so that $\Psi^-_{\eps}(x)=0$.

For each small $\eps>0$, put $d(\eps) := \Psi^+ F^n_{\eps}(x)$, where $n=n(\eps)$ is the smallest integer number  such that $F^n_{\eps}(x) \in [x+1, F_{\eps}(x)+1]$.

Consider analytic charts $\chi_{\eps}\colon [x,F_{\eps}(x)]/F_{\eps}\to \bbR/\bbZ$ that converge to the chart $(\Psi^+(z+1)\mod 1)$ on $[x,F(x)]/F$ as $\eps\to 0$; the convergence is uniform in a neighborhood of $[x,F_{\eps}(x)]$ in $\bbC$.  For example, we may take perturbed Fatou coordinates of $F_{\eps}$, see \cite[Proposition 3.2.2, coordinates $\Phi_{\pm,f}$]{Shi}.

\begin{lemma}
\label{lem-R-in-Fatou-coord}
Under assumptions of Theorem \ref{th-main-1}, suppose that  the sequence $\eps_k\to 0$, $\eps_k>0$, satisfies $d(\eps_k)\to c$. Then $\chi_{\eps} (\mathcal R f_{\eps})\chi_{\eps}^{-1}$ tends to ${\mathbf K}+c$ as $k\to \infty$, uniformly on some neighborhood of $\bbR/\bbZ$ in $\bbC/\bbZ$.
\end{lemma}

\begin{proof}
The definition of $\mathcal R f_{\eps}$ suggests that we study the maps $F_{\eps_k}^{n(\eps_k)}-1$ and $F_{\eps_k}^{n(\eps_k)-1}-1$ on $[x, F_{\eps}(x)]$.

Since $d(\eps_k)\to c$, we have that
\begin{equation}
\label{eq-F-Psi}
 \lim_{k\to\infty}F_{\eps_k}^{n(\eps_k)}(x)=(\Psi^{+})^{-1}(c).
\end{equation}
Note that $L_c(x)=(\Psi^{+})^{-1} T_c \Psi^-(x)=(\Psi^{+})^{-1}(c)$ due to our normalization $\Psi^{-}(x)=0$, so the right-hand side of \eqref{eq-F-Psi} equals $L_c(x)$. Due to Lavaurs theorem (Theorem \ref{th-Lav-1}), in some neighborhood of $[x, F_{\eps}(x)]$  in $\bbC$, the maps $F^{n(\eps_k)}_{\eps_k}$ tend uniformly  to the Lavaurs map $L_c$. So the maps  $F^{n(\eps_k)}_{\eps_k}-1$ converge uniformly  to $L_c-1$  in some neighborhood of $[x, F_{\eps}(x)]$, and the maps $F^{n(\eps_k)-1}_{\eps_k}-1$ converge uniformly  to  $L_{c-1} -1 = F^{-1}L_c-1$.
Due to Remark \ref{rem-Lav-chart}, in the chart $\Psi^+(z+1)$ on $[x, F(x)]$, the map $L_c-1$  equals ${\mathbf K}+c$. So in the chart $(\Psi^+(z+1)\mod 1)$ on $[x, F(x)]/F$, the  maps $L_c-1$ and $L_{c-1}-1$ equal ${\mathbf K}+c$.
% $z\mapsto \Psi^+(\Psi^{+})^{-1} T_d \Psi^-((\Psi^+)^{-1}-1) = T_d {\mathbf K} ={\mathbf K}+d$.
Therefore, in any analytic charts on $[x, F_{\eps}(x)]/F_{\eps}$ that tend to $(\Psi^+(z+1)\mod 1)$, we have the following uniform convergence in a neighborhood of $\bbR/\bbZ$:
\begin{align*}
&\chi_{\eps} (F_{\eps_k}^{n(\eps_k)}-1) \chi_{\eps}^{-1}\to {\mathbf K}+c\\
&\chi_{\eps} (F_{\eps_k}^{n(\eps_k)-1}-1) \chi_{\eps}^{-1}\to {\mathbf K}+c.
\end{align*}
 Since $\mathcal R f_{\eps}$ is induced by $F_{\eps_k}^{n(\eps_k)}-1$ on the one subsegment of $[x, F_{\eps}(x)]/F_{\eps}$ and by $F_{\eps_k}^{n(\eps_k)-1}-1$ on the other, the result follows.

\end{proof}

The function $d(\cdot)$ defined above is almost suitable as a parametrization of bubbles $B_{a_n, F_{\eps}}$  needed for Theorem \ref{th-main}. The following two lemmas study this function.

\begin{lemma}
\label{lem-d}
The function $d(\cdot)$ is monotonic and continuous on $I_{r, F_{\eps}}$ if  $r \neq 1/n$. On $I_{1/n, F_{\eps}}$ it is monotonic whenever continuous and has a jump.
The size of the jump tends to $1$ from above as $n\to\infty$.
In any case, the values of $d(\eps)$ for $\eps$ small belong to a small neighbohood of $[0,1]$.

\end{lemma}
\begin{proof}
The last claim is clear because $\Psi^+([x+1,F_{\eps}(x)+1])$ is close to $\Psi^+([x+1, F(x)+1]) = [0,1]$.

Since $F^{n}_{\eps}$ is monotonic, the function $d(\cdot)$ is monotonic whenever continuous. It has jumps at the points $\eps_n$ with $F^{n}_{\eps_n}(x)=x+1$, i.e. where $\mathcal Rf_{\eps_n}(x)=x$. So the jump points belong to the set $\{\rot \mathcal Rf_{\eps}=0\}$, i.e. to $I_{1/n, F_{\eps}}$.

The size of the jump is $\Psi^+ F_{\eps_n}(x)-\Psi^+(x)$, which tends to $\Psi^+ F(x)-\Psi^+(x)=1$ from above as $\eps_n\to 0$.
\end{proof}

Recall that $a_n = -1/(p/q-n)$, as in Theorem \ref{th-main-1}. The following lemma shows that $d(\cdot)$ takes the segments  $I_{a_n, F_{\eps}}$ that parametrize bubbles $B_{a_n, F_{\eps}}$ (approximately) to the segments $I_{p/q, \mathbf K+\omega}$ that parametrize $B_{p/q, \mathbf K+\omega}$.
\begin{lemma}
\label{lem-d-monot}
 Put $I_n:=d(I_{a_n, F_{\eps}})\mod 1 \subset \bbR/\bbZ$. Then the set of limit points of $I_n$ as $n\to\infty$ coincides with $I_{p/q, {\mathbf K}+\omega}\mod 1$.
\end{lemma}
\begin{proof}
Let $d$ be a limit point of $I_n$, and let us prove that  $d\in I_{p/q, {\mathbf K}+\omega}$.
 Take $\eps_k\in I_{a_{n_k}, F_{\eps}}$ such that $d(\eps_k)\to d$; then due to Lemma \ref{lem-R-in-Fatou-coord}, $\chi_{\eps_k} (\mathcal R f_{\eps_k})\chi_{\eps_k}^{-1} \to {\mathbf K}+d$. Note that $\rot \chi_{\eps_k} (\mathcal R f_{\eps_k})\chi_{\eps_k}^{-1}= \rot \mathcal R f_{\eps}=-1/\rot F_\eps$, so for $\eps_k\in I_{a_{n_k}, f_\eps}$, we have $\rot \chi_{\eps_k} (\mathcal R f_{\eps_k})\chi_{\eps_k}^{-1} \equiv -1/a_{n_k} \equiv p/q \pmod 1$.  This implies  $\rot ({\mathbf K}+d)=p/q$, hence  $d\in I_{p/q, {\mathbf K}+\omega}$.

If $I_{p/q, {\mathbf K}+\omega}$ is just one point, the proof is finished. If $I_{p/q, {\mathbf K}+\omega}$  is a segment, we also need to prove that any point of $I_{p/q, {\mathbf K}+\omega}$ is a limit point of   $I_n$.
Take $c \in I_{p/q, {\mathbf K}+\omega}$ such that ${\mathbf K}+c$ is hyperbolic. This holds for all points of $I_{p/q, {\mathbf K}+\omega}$ except a finite set. Let $\eps_k\to 0 $ be such that $d(\eps_k)\equiv c \pmod 1$; the existence of such an infinite sequence $\eps_k$ follows from the definition of $d$. Then $\chi_{\eps_k} (\mathcal R f_{\eps_k})\chi_{\eps_k}^{-1}$ tends to the hyperbolic map ${\mathbf K}+c$ due to Lemma \ref{lem-R-in-Fatou-coord}, so  $\rot (\chi_{\eps_k} (\mathcal R f_{\eps_k})\chi_{\eps_k}^{-1}) = p/q$ for large $k$. This implies $\eps_k \in \bigcup I_{a_n, F_{\eps}}$ as explained above. Finally, $c$ belongs to the set of limit points of $I_n$.

Since the set of limit points is closed, $I_{p/q, {\mathbf K}+\omega}$ belongs to the set of limit points of $I_n$.

\end{proof}

\subsection{Parabolic renormalization at a rational point}
\label{sec-egg-kl}
Here we prove the analogues of the results from the previous section for the case $\rot f = k/l$.

Let $f$ be a circle diffeomorphism having one parabolic cycle of period $l$, $f^l(0)=0$. The analogue of Lemma \ref{lem-R-in-Fatou-coord} is the following. Recall that we normalize  Fatou coordinates of $f^l$ at $0$ by $\Psi^-(x)=0$, $\Psi^+(y)=0$ where $x \in (a,0), y := f^r(x)\in (0,b)$.
For each small $\eps>0$, put  $\tilde d(\eps):=\Psi^+ f^{nl}_{\eps}(x)$, where $n=n(\eps)$ is the smallest number such that $f^{nl-r}_{\eps}(y) = f^{nl}_{\eps}(x) \in [y, f_{\eps}^l(y)]$. Due to  Lemma \ref{lem-firstreturn}, this is the first point of the orbit of $y $ that belongs to $[y, f_{\eps}^l(y)]$. Let $\mathcal R^{k/l} f_{\eps}$ be $k/l$-renirmalizations associated with fundamental domains $[y, f^l_{\eps}(y)]$.

Note that $(\Psi^+ \mod 1)$ defines an analytic chart on $[y,f^l(y)]/f^l$.
Consider analytic charts $\tilde \chi_{\eps}$ on the circles $[y,f^l_{\eps}(y)]/f^l_{\eps}$ that converge to the chart $(\Psi^+ \mod 1)$ on $[y,f^l(y)]/f^l$, and the convergence is uniform on some neighborhood of $[y, f^l(y)]$ in $\bbC$. Again, we may take perturbed Fatou coordinates of $f_{\eps}^l$.

\begin{lemma}
\label{lem-R-in-Fatou-coord-kl}
Under assumptions of Theorem \ref{th-main-1-kl}, suppose that the sequence $\eps_k\to 0, \eps_k>0$, is such that $\tilde d(\eps_k)\to c$. Then $\tilde \chi_{\eps} (\mathcal R^{k/l} f_{\eps})\tilde \chi_{\eps}^{-1}$ tends to $\tilde {\mathbf K}+c$ as $k\to \infty$, uniformly on some neighborhood of $\bbR/\bbZ$ in $\bbC/\bbZ$.

\end{lemma}
\begin{proof}
Since $d(\eps_k)\to c$, we have that $f_{\eps_k}^{ln(\eps_k)}(x)$ tends to $(\Psi^{+})^{-1}(c)$, which equals $L_c(x)$, where $L_c=(\Psi^{+})^{-1}T_c (\Psi^{-})$ is a Lavaurs map.   Due to Lavaurs theorem (see Theorem \ref{th-Lav-1}), in some neighborhood of $[x, f^l_{\eps}(x)]$  in $\bbC$, the maps $f^{ln(\eps_k)}_{\eps_k}$ tend uniformly  to the  Lavaurs map $L_c$.

So $f^{ln-r}_{\eps_k}$ on some neighborhood of $[y, f^l(y)]$ converge uniformly  to $L_cf^{-r}$. Similarly, the maps $f^{l(n-1)-r}_{\eps_k}$ converge uniformly  to $L_{c-1}f^{-r}$.
Due to Remark \ref{rem-Lav-chart-kl}, in the chart $(\Psi^+ \mod 1)$ on $[y, f^l(y)]/f^l$, the map $L_c f^{-r}$ equals $\tilde {\mathbf K} +c$.
So in any analytic charts $\tilde \chi_{\eps}$ on $[y, f_{\eps}^l(y)]/f_{\eps}^l$ that tend uniformly to $(\Psi^+(z) \mod 1)$, we have the uniform convergence $\tilde \chi_{\eps} f_{\eps_k}^{ln(\eps_k)-r} \tilde\chi_{\eps}^{-1}\to \tilde{\mathbf K}+c$ and $\tilde\chi_{\eps} f_{\eps_k}^{l(n(\eps_k)-1)-r} \tilde\chi_{\eps}^{-1}\to \tilde{\mathbf K}+c$. Due to Lemma \ref{lem-renorm-kl}, the map $\mathcal R^{k/l}(f_{\eps})$ is induced by $f_{\eps_k}^{ln(\eps_k)-r}$ on the one subsegment of $[y,f_{\eps}^l(y)]$ and by $f_{\eps_k}^{l(n(\eps_k)-1)-r}$ on the other, and the result follows.

\end{proof}

In the next two lemmas, we study the function $\tilde d(\cdot)$. We will use this function to reparametrize bubbles for Theorem \ref{th-main-kl}. Recall that $r, s$ are such that $kr+ls=1$. Recall that  $\tilde R(z)= (rz+s) / (-lz+k)$, and this is the mapping that acts on rotation numbers when we make $k/l$-renormalizations.
\begin{lemma}
\label{lem-d-kl}
The function $\tilde d(\cdot)$ is monotonic and continuous on $I_{r, F_{\eps}}$ if  $\tilde R(r) \neq 0 \pmod 1$. Otherwise,  it is monotonic whenever continuous on  $I_{r, F_{\eps}}$ and has a jump.
The size of the jump tends to $1$ from above as $n\to\infty$.
In any case, the values of $\tilde d(\eps)$ for $\eps$ small belong to a small neighborhood of $[0,1]$.
\end{lemma}
\begin{proof}
 Note that the jump of $\tilde d$ occurs when $f^{nl}_{\eps}(x)= f^r_{\eps}(x)$, i.e. $\mathcal R^{k/l} f_{\eps}$ has a fixed point at $x$. This happens when $\rot \mathcal R^{k/l}f_\eps=0$, i.e. $\tilde R(\rot f_{\eps})=0$. The rest of the proof repeats the proof of Lemma \ref{lem-d}.
\end{proof}

 Recall that $p/q$ is a rational number, and the sequence $\{a_n\}$ is such that  $\tilde R(a_n)= p/q-n$.
\begin{lemma}
\label{lem-d-monot-kl}
 Put $\tilde I_n:=\tilde d(I_{\tilde a_n, F_{\eps}})\mod 1 \subset \bbR/\bbZ$. Then the set of limit points of $\tilde I_n$ as $n\to\infty$ coincides with $I_{p/q, \tilde {\mathbf K}+\omega}\mod 1$.
\end{lemma}
The proof is completely analogous to the proof of Lemma \ref{lem-d-monot}.

\subsection{Continuity of complex rotation numbers}
The following lemma shows that $\bar \tau$ is continuous in the metric of uniform convergence at a circle diffeomorphism with at most one parabolic cycle.
\begin{lemma}
\label{lem-tau-contin}
 Let $g_{k}$ be a sequence of analytic circle diffeomorphisms that converges to an analytic circle diffeomorphism $g$, uniformly  on some neighborhood of $\bbR/\bbZ$ in $\bbC/\bbZ$. Suppose that $g$  has at most one parabolic cycle. Then $ \bar \tau (g_{k}) \to  \bar \tau (g)$.
\end{lemma}
The proof of this lemma constitutes Sec. \ref{sec-contin}.
\subsection{Proof of Theorems \ref{th-main-1} and \ref{th-main} modulo auxiliary lemmas}
\label{sec-proofs}
The most essential components of the proofs of Theorems \ref{th-main-1} and \ref{th-main} are: the behaviour of complex rotation numbers under renormalizations (Lemma \ref{lem-renorm}), the convergence of renormalizations to $\mathbf K+c$ (Lemma \ref{lem-R-in-Fatou-coord}), and the continuity of complex rotation numbers (Lemma \ref{lem-tau-contin}). Their proofs are contained in Sections \ref{sec-renorm-proof}, \ref{sec-egg}, and \ref{sec-contin} respectively. The proofs in this section are straightforward and show how to reduce the theorems to the lemmas mentioned above.

% Informally, the idea behind the proofs is the following. Making  renormalizations of a circle diffeomorphism, we apply the map  $R(z)=-1/z$ to the complex rotation number (Lemma \ref{lem-renorm}), i.e. to the bubbles. On the other hand, the renormalizations of our family, $\mathcal Rf_{\eps}$,  in suitable charts, are close to the family $\mathbf K+c$ (Lemma \ref{lem-R-in-Fatou-coord}). Since the complex rotation number does not depend on the charts (Lemma \ref{lem-chart}) and is continuous except some degenerate cases (Lemma \ref{lem-tau-contin}), the images of bubbles of the initial family $f_{\eps}$ under $R$ are close to the bubbles of $\mathbf K+c$. An appropriate correspondence of bubbles follows from Lemma \ref{lem-d-monot}.

The following lemma
% is the formalization of the text above; it
is a common part of the proofs of Theorems \ref{th-main-1} and \ref{th-main}. We recall that $R(z)=-1/z$ and $d(\cdot)$ is defined in Sec. \ref{sec-egg}.

\begin{lemma}
\label{lem-summ}
Suppose that $c\in I_{p/q, \mathbf K+\omega}$ and $\mathbf K+c$ has at most one parabolic fixed point. Let $\eps_k\to 0$ be such that $(d(\eps_k)\mod 1) \to c$.

Then $R(\bar \tau (F_{\eps_k}))\to \bar \tau (\mathbf K+c)$ in $\overline \bbH/\bbZ$.
\end{lemma}
\begin{proof}

When we make  renormalizations, we apply the map  $R(z)=-1/z$ to complex rotation numbers (Lemma \ref{lem-renorm}):
\begin{equation}
\label{eq-R}
 R( \bar \tau(F_{\eps})) \equiv \bar \tau(\mathcal R f_{\eps}) \pmod 1.
\end{equation}
The renormalizations $\mathcal Rf_{\eps}$  in suitable charts are close to the family $\mathbf K+\omega$ due to Lemma \ref{lem-R-in-Fatou-coord}. Formally, take the charts $\chi_{\eps}$ as in Lemma \ref{lem-R-in-Fatou-coord}. Since the complex rotation number does not depend on the charts (Lemma \ref{lem-chart}), we have
\begin{equation}
\label{eq-R-2}
 \bar \tau(\mathcal R f_{\eps}) =\bar \tau(\chi_{\eps}(\mathcal R f_{\eps})\chi_{\eps}^{-1}).
\end{equation}
Due to Lemma \ref{lem-R-in-Fatou-coord}, the  maps $\chi_{\eps_n} (\mathcal R f_{\eps_n})\chi_{\eps_n}^{-1}$ tend to  $\mathbf K+c$ uniformly on some neighborhood of $\bbR/\bbZ$ in $\bbC/\bbZ$. Using \eqref{eq-R}, \eqref{eq-R-2}, and the continuity of $\bar \tau$ at $\mathbf K+c$  (Lemma \ref{lem-tau-contin}), we get $R( \bar \tau(F_{\eps_n})) = \bar \tau(\chi_{\eps_n} (\mathcal R f_{\eps_n})\chi_{\eps_n}^{-1})\to \bar \tau(\mathbf K+c)$.
\end{proof}

\begin{proof}[Proof of Theorem \ref{th-main-1}]
 Recall that we consider $a_n$-bubbles of $F_{\eps}$, $B_{a_n, F_{\eps}}$, and the map $R(z)=-1/z$
takes the sequence $\{a_n\}$ to the sequence $\{p/q-n\}$.

First, prove that the set of limit points of $R(B_{a_k, F_{\eps}})$ contains $B_{p/q, \mathbf K +\omega}\smallsetminus \bbR$.

Consider any $c\in I_{p/q, \mathbf K+\omega}$ such that  ${\mathbf K}+c$ is hyperbolic; this corresponds to an arbitrary point of $B_{p/q, \mathbf K +\omega}\smallsetminus \bbR$. Recall that the limit points of the sequence of segments $d(I_{a_n, F_{\eps}})$ form the segment $I_{p/q, \mathbf K+\omega}$ (Lemma \ref{lem-d-monot}), so we may fix $\eps_n\to 0, \eps_n\in \bigcup I_{a_k, F_{\eps}}$, such that $d(\eps_n)\to c$. Due to Lemma \ref{lem-summ}, we have that $R(\bar \tau(f_{\eps_n}))\to \bar \tau (\mathbf K+c)$.

So the  set of limit points of  $R(B_{a_k, F_{\eps}})$ contains $ B_{p/q, \mathbf K +\omega}\smallsetminus \bbR$.

Since the set of limit points is closed, it contains $ B_{p/q,{\mathbf K}+\omega}$ as well.

\end{proof}

% Define $R_{\eps}=\chi_{\eps} (\mathcal R f_{\eps})\chi_{\eps}^{-1}$ as in Lemma \ref{lem-R-in-Fatou-coord}. Due to Lemma \ref{lem-chart}, $\bar \tau(\mathcal R f_{\eps})  = \bar \tau(R_{\eps})  $. So we have
% \begin{equation}
% \label{eq-curve}
% R(B_{a_n, F_{\eps}}) = \{\bar\tau(R_{\eps}), \eps \in I_{a_n, F_{\eps}}\}
% \end{equation}

\begin{proof}[Proof of Theorem \ref{th-main}.]
 In this theorem, we must present suitable parametrizations of bubbles $B_{a_n, F_{\eps}}$.

 First, consider the case $p/q\neq 0$, i.e. $a_n$ is not the sequence $1/n$; in this case, the suitable parametrization will be the function $d(\cdot)$ introduced in Sec.~\ref{sec-egg}. Due to Lemma \ref{lem-d},  $d(\cdot)$ is a continuous monotonic function on $I_{a_n, F_{\eps}}$.

 Suppose that the curves $R(B_{a_n, F_{\eps}})$ parametrized by $d(\cdot)$ do not tend uniformly to $B_{p/q, \mathbf K+c}$ parametrized by $c$. So for some $\delta>0$, there exists a sequence $\eps_k \to 0, \eps_k \in \bigcup I_{a_{n}, F_{\eps}}$, such that
 \begin{equation}
 \label{eq-bubble-dist}
  \dist (R(\bar\tau( f_{\eps_k})), \bar \tau (\mathbf K + d(\eps_k)))>\delta.
 \end{equation}
Extracting subsequences, we may assume that $d(\eps_k)$ converges to some $c$. Recall that the limit points of the segments $d(I_{a_n, F_{\eps}})$ form the segment $I_{p/q, \mathbf K+\omega}$ (Lemma \ref{lem-d-monot}), hence $c\in I_{p/q, \mathbf K+\omega}.$

Now Lemma \ref{lem-summ} implies that $R(\bar\tau( f_{\eps_k})) \to \bar \tau(\mathbf K +c)$. Since $\bar \tau(\mathbf K +c)$ is continuous with respect to $c$, the distance in \eqref{eq-bubble-dist} must tend to zero, and we get a contradiction.

The contradiction shows that the curves $R(B_{a_k, F_{\eps}})$, parametrized by $d(\cdot)$,  tend uniformly to $B_{p/q, \mathbf K+\omega}$, q.e.d.

If $p/q=0$, i.e. $a_n=1/n$, the function $d(\cdot)$ is not continuous on $I_{a_n, F_{\eps}}$. It has a jump  on each segment $I_{1/n, F_{\eps}}$. However  the size of the jump of $(d(\cdot) \mod 1)$ tends to zero as $n\to \infty$, see  Lemma \ref{lem-d}. So there exists a continuous monotonic reparametrization $D(\cdot) \colon I_{1/n, F_{\eps}}\to \bbR$ of the curves $R(B_{a_n, F_{\eps}})$ such that $(D(\eps) - d(\eps))\mod 1 \to 0$ as $\eps \to 0$. Then for any sequence $\eps_k\to 0$, we have $d(\eps_k)\to c$ on $\bbR/\bbZ$ iff $D(\eps_k)\to c$. So Lemma \ref{lem-summ} holds true for $D(\cdot)$. The same arguments as above imply the statement of Theorem \ref{th-main}.
\end{proof}

Exactly the same arguments lead to the following statement. It seems to be the strongest possible statement in assumptions of Theorem \ref{th-main-1} only.

\begin{theorem}
\label{th-main-1.5}
 In assumptions of Theorem \ref{th-main-1}, let $Z$ be a finite set of values  $c\in I_{p/q, \mathbf K+\omega}$ such that $\mathbf K+c$ has at least two parabolic fixed points. Let $U$ be an arbitrarily small neighborhood of $Z$. Then the  curves $\{R(\bar \tau(f_{\eps}))\mid \eps \in I_{a_n, F_{\omega}}\smallsetminus d^{-1}(U)\}$, parametrized by $d(\cdot)$ or $D(\cdot)$ as in Theorem \ref{th-main}, tend uniformly to $\{\bar \tau(\textbf K+c)\mid c \in I_{p/q, \mathbf K+\omega}\smallsetminus U\}$.
\end{theorem}

\subsection{Proof of Theorems \ref{th-main-1-kl} and \ref{th-main-kl} modulo auxiliary lemmas}

The proofs are completely analogous. We should make the following replacements:
\begin{itemize}
\item $\tilde d(\cdot)$ replaces $d(\cdot)$, $\tilde{R}$ replaces $R$,  $\tilde {\mathbf K}$ replaces $\mathbf K$, $\mathcal R^{k/l}$ replaces $\mathcal R$, $\tilde a_n$ replaces $a_n$.

\item Instead of Lemmas \ref{lem-renorm}, \ref{lem-R-in-Fatou-coord}, \ref{lem-d}, and \ref{lem-d-monot}, we should use their direct analogues Lemmas  \ref{lem-renorm-1},  \ref{lem-R-in-Fatou-coord-kl}, \ref{lem-d-kl}, and \ref{lem-d-monot-kl}. Their proofs are contained in Sec.\ref{sec-renorm-proof} and \ref{sec-egg-kl}.

\item In the proof of Theorem \ref{th-main-kl}, the exceptional sequence  $\{\tilde a_n\}$ is $\tilde R^{-1}(\bbZ)$ instead of $\{1/n\}$.
\end{itemize}
All the rest of the proofs is literally the same.

% Recall that $n(\eps)$ is the number of iterates of $F_{\eps}$ we need to get from $x$ to $[y, F_{\eps}(y)]$. Let $\tilde d({\eps}) = \Psi^+_{0} F^{n(\eps)}_{\eps}(x)\mod 1$. This function is uniformly close to $d(\eps)$, because $\Psi^+_{\eps}$ tends to $\Psi^+_{0}$ pointwise on some neighborhood of $[y, f_{\eps}(y)]$, thus uniformly on $[y, f_{\eps}(y)]$. This function is monotonic because so is $F_{\eps}$.

% Now, due to  Lemma \ref{lem-Fatou-contin}, $\Psi_{\eps}^{-}\circ (\Psi_{\eps}^{+})^{-1} $ tends to $\Psi_{0}^{-}\circ (\Psi_{0}^{+})^{-1} $ uniformly on some neighborhood of $\bbR/\bbZ$, and $d(\eps)$ is
% Reference to Lemma \ref{lem-d-c} finishes the proof. \fixme{More details}

\section{Complex rotation numbers for hyperbolic diffeomorphisms}
\label{sec-Buff}
Let $g$ be a hyperbolic analytic circle diffeomorphism, let $G$ be its lift to the real line.
In this section, we present more explicit construction of a complex rotation number $\bar \tau(G)$. Namely, we will construct a non-degenerate torus $\mathcal E(g)$ with modulus $\bar \tau(G)$.
The construction was suggested by X.Buff and used in  \cite{NG2012-en}, \cite{XB_NG}.

Informally, we construct a special fundamental domain of $g$ and let $\mathcal E(g)$ be the quotient of this domain by the action of $g$. This fundamental domain is an annulus in $\bbC/\bbZ$ with curvilinear boundaries, it is close to $\bbR/\bbZ$  and passes above repelling cycles and below attracting cycles of $g$.

\begin{figure}
\begin{center}
 \includegraphics[width=0.6\textwidth]{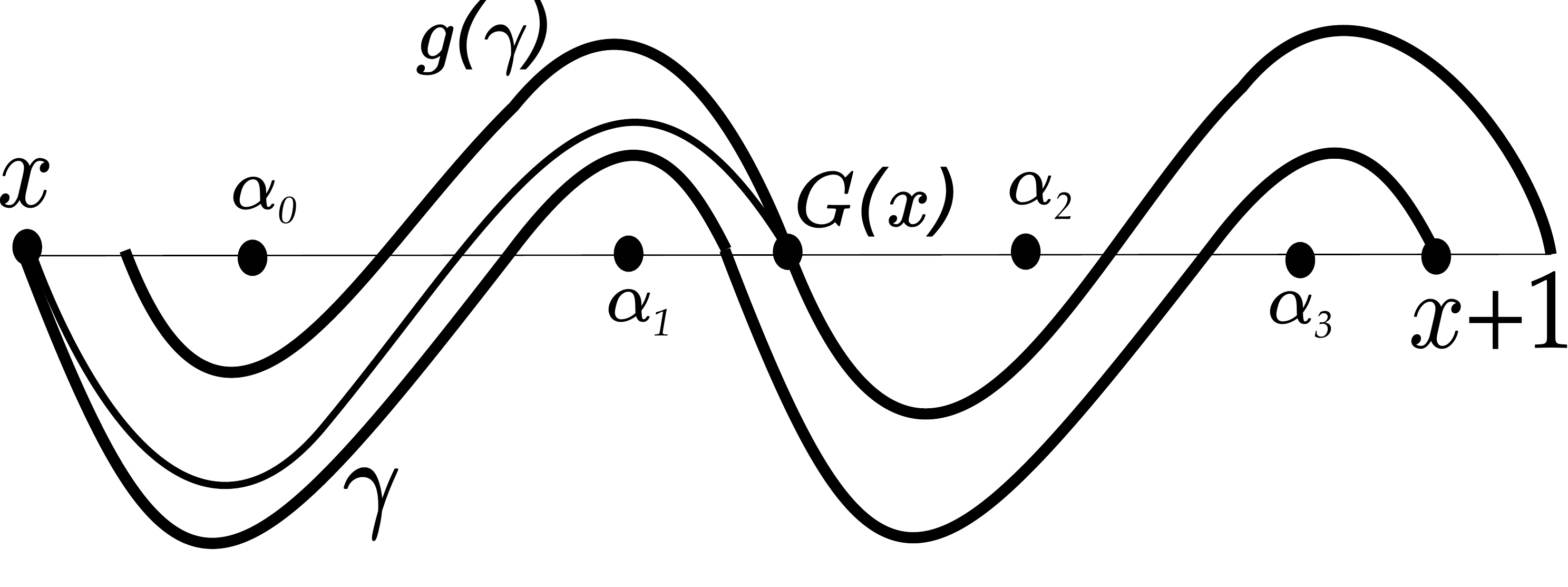}
\end{center}
\caption{Construction of $\mathcal E(g)$ and its generators for the circle diffeomorphism $g$ with two 2-periodic hyperbolic orbits}\label{fig-buff}
\end{figure}

In more detail, suppose that $g\colon\bbR/\bbZ\to \bbR/\bbZ$ has rational rotation number $p/q$ and is hyperbolic. Let $m\geq 1$ be the number of attracting cycles of $g$; it is equal to the number of repelling cycles.
Let $\alpha_j$, $j\in \bbZ/(2mq)\bbZ$, be the periodic points of $g$, ordered cyclically; even indices correspond to attracting periodic points and odd indices to repelling periodic points. Then $g(\alpha_j)=\alpha_{j+pm}$.

Let $\rho_j:= (g^{\circ q})'(\alpha_j)$. Let $\phi_j\colon(\bbC,0)\to (\bbC/\bbZ,\alpha_j)$ be the uniformizing chart for  $g^{\circ q}$, i.e. $ g^{\circ q}\circ \phi_j(z)=\phi_j(\rho_j z)$, and normalize the charts so that $\phi'_j(0)=1$.
Then $\phi_j$ extends univalently to a neighborhood of the real axis, and its range contains a neighborhood of $ (\alpha_{j-1},\alpha_{j+1}).$

For each $j\in \bbZ/(2mq)\bbZ$, let $x_j$ be a point in $(\alpha_j,\alpha_{j+1})$. Let
 $C_j$ be an arc of a circle with endpoints  $\phi_j^{-1}(x_{j-1})$, $\phi_j^{-1}(x_{j})$, such that $C_j$ is close to the real axis and located above $\bbR$ for odd $j$ and below $\bbR$ for even $j$. Put  $\gamma=\bigcup \phi_j(C_j)$. We  choose $x_j, C_j$ so that $g(\gamma)$ is above $\gamma$ in $\bbC/\bbZ$. This requires achieving $g(x_j) \in (\alpha_{j+pm}, x_{j+pm})$ for attracting $\alpha_j$ and $g(x_j) \in (x_{j+pm}, \alpha_{j+pm+1})$ for repelling $\alpha_j$, and the same for the heights of $C_j$. See \cite[Sec. 5]{XB_NG} for more details.

Then, $\gamma$ is a simple closed curve in $\bbC/\bbZ$, $g$ is univalent in a neighborhood of $\gamma$, and  $g(\gamma)$ lies above $\gamma$ in $\bbC/\bbZ$; see Fig. \ref{fig-buff}.
% For $\omega$ sufficiently close to $0$, the curve $f_\omega(\gamma)=f(\gamma)+\omega$ remains above $\gamma$ in $\bbC/$.
The curves $\gamma$ and $g(\gamma)$ bound the annulus $\Pi \subset \bbC/\bbZ$. Glueing its two sides  via $g$, we obtain the complex torus $\mathcal E (g)$.

If the lift $G$ of $g$ is fixed, this torus has two distinguished generators of the first homology group: the first generator is $\gamma$, and the second one is a curve that joins $x$ to $ G(x)$ in the lift of $\Pi\subset \bbC/\bbZ$ to $\bbC$, see Fig. \ref{fig-buff}.
The modulus of $\mathcal E (g)$ is a unique $\tau=\tau(\mathcal E (g)) \in \bbH$ such that $\mathcal E (g)$ is biholomorphically equivalent to $\bbC/(\bbZ+\tau\bbZ)$ with its first generator corresponding to  $\bbR/\bbZ$ and the second generator corresponding to $\tau\bbR/\tau\bbZ$. The modulus of
$\mathcal E (g)$ depends on the choice of a lift $G$ in the same way as described in Sec. \ref{sec-lift}.
% As mentioned above, the modulus $\tau$ of $\mathcal E(g)$ equals the complex rotation number $\bar \tau(g)$ from Theorem \ref{th-XB_NG}; actually, this construction was used as a definition of $\bar \tau(g)$ in \cite[Sec. 5]{XB_NG}.

The following theorem is contained in \cite{XB_NG}.
\begin{theorem}
\label{th-Eg=tau}
 For each hyperbolic circle diffeomoprhism $g+\omega$, the modulus $\tau(\mathcal E(g+\omega))$ equals the complex rotation number $\bar \tau_G(\omega)$.
  \end{theorem}
  Actually, the construction of $\mathcal E(g+\omega)$   was used as a definition of $\bar \tau_g(\omega)$ in \cite[Sec. 5]{XB_NG}, so this theorem follows from the result of \cite{XB_NG}.
  \begin{remark}
  \label{rem-depend-on-g+omega}
  Since $\bar \tau_g(\omega)$ equals the modulus of $\mathcal E(g+\omega)$, it only depends on $g+\omega$, which motivates our notation $\bar \tau(g+\omega) = \tau_g(\omega)$.

  Moreover, $\mathcal E(g)$ does not depend on the analytic chart on the circle. Thus for two analytically conjugate diffeomorphisms $g_1, g_2$, we have $\mathcal E(g_1)=\mathcal E(g_2)$, which implies Lemma \ref{lem-chart}.
  \end{remark}

  \begin{remark}
  \label{rem-analyticity}
  Due to \cite[Sec. 2.1, Proposition 2]{Ris}, the modulus $\tau$ of $\mathcal E(g)$ depends holomorphically on $g$.
%   Here we sketch the proof.

%   Roughly speaking, this follows from the fact that the gluings that produce $\mathcal E(g)$ depend holomorphically on $g$. To give a formal proof, one should introduce  a conformal structure on the standard torus $\bbC/\bbZ+i\bbZ$ that makes it conformally equivalent to $\mathcal E(g)$. This conformal structure depends analytically on $g$. Now the Ahlfors-Bers theorem implies that the  uniformization $H$ of $\bbC/\bbZ+i\bbZ$ with this conformal structure depends analytically on $g$. The modulus of $\mathcal E(g)$ equals $H(i)$, so it depends holomorphically on $g$.
  \end{remark}

%   \begin{theorem}
%    The modulus $\tau$ of $\mathcal E(g)$ depends holomorphically on $g$.
%   \end{theorem}

\section{Complex rotation numbers under renormalizations.}
\label{sec-renorm-proof}

Here we prove Lemmas \ref{lem-renorm} and \ref{lem-renorm-1}.
\begin{proof}[The proof of Lemma \ref{lem-renorm}]
 If $\rot g$ is irrational, then $\bar \tau (g)=\rot g$ due to Theorem \ref{th-XB_NG}. Since $\rot (\mathcal R g) \equiv \frac{-1}{ \rot g} \pmod 1$, the number  $\rot (\mathcal R g)$ is also irrational, and
 \begin{equation}
 \label{eq-tau-rot}
 \bar \tau (\mathcal R  g) = \rot (\mathcal R g) =  \frac{-1}{ \rot (G)}  = \frac{-1}{ \bar \tau(G)}
 \end{equation}
 q.e.d. Similarly, if $g$ has a parabolic cycle, then so does $\mathcal R g$, and we have \eqref{eq-tau-rot} again. The only remaining case is the case of a hyperbolic $g$.

Take any point  $x$ that is not  a periodic point of $g$.
Fix $n\in \bbN$ such that the first return map to $[x,g(x)]$ under iterates of $g$ is $g^n$ on the one subseqment of $[x, g(x)]$ and  $g^{n-1}$ on the other.

Consider the complex torus $\mathcal E(g)=\Pi/g$ defined in the previous section; recall that $\Pi$ is the curvilinear annulus between   $\gamma$ and $g(\gamma)$. Clearly, we may assume that $\gamma$ passes through $x$, and take $\gamma$ sufficiently close to $\bbR/\bbZ$ such that $g^n$ is defined and univalent in the wider annulus between $\gamma$ and $g^{n+1}(\gamma)$. Let this annulus be $\Pi^{(n+1)}$. Let the lifts of the annuli $\Pi, \Pi^{(n+1)}$ to $\bbC$
be denoted by $\tilde \Pi, \tilde \Pi^{(n+1)}.$

 Put $\tau:=\bar \tau(G)$ for shortness. Due to Theorem \ref{th-Eg=tau}, $\mathcal E(g)$ is biholomorphic to $\bbC / \bbZ+\tau \bbZ$. The biholomorphism $H \colon \mathcal E(g) \to \bbC / \bbZ+\tau \bbZ$ lifts to the map $\hat H \colon \tilde \Pi \to \bbC$ that conjugates $G$ to $z\mapsto z+\tau$ and the shift by $1$ to itself. It extends, via iterates of $G$, to the strip $\tilde \Pi^{(n+1)}$. Now $\hat H \colon \tilde \Pi^{(n+1)} \to \bbC$ conjugates $G$ to the shift by $\tau$. Below we show that the same map $\hat H$ rectifies the complex torus $\mathcal E(\mathcal R g)$.

 First of all, introduce a suitable construction of $\mathcal E(\mathcal R g)$; this will be a particular case of the general construction from Sec. \ref{sec-Buff}.

 \begin{figure}
 \includegraphics[width=0.7\textwidth]{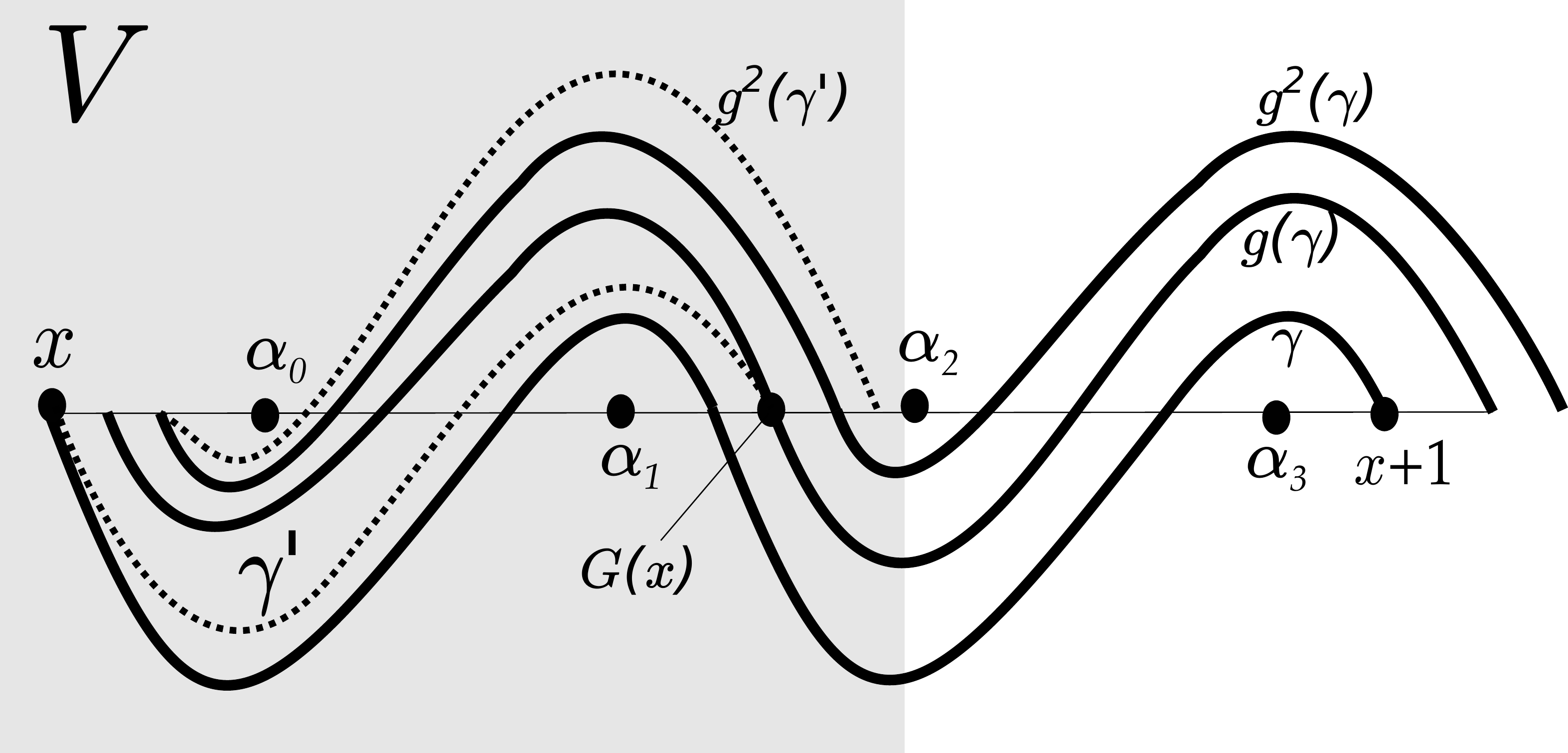}
  \caption{Construction of $\mathcal E(\mathcal R g)$ for $g$ having two $2$-periodic orbits. The curves $\gamma, g(\gamma), g^2(\gamma)$ are thick, the curves $\gamma'$ and $g^2(\gamma')$ are dotted, $V$ is shadowed}\label{fig-Buff-Rg}
 \end{figure}

Take a small neighborhood $V$ of $[x,G(x)]$ in which  $G$ is univalent.
Let a curve $\gamma'\subset V$ be the union of arcs of circles in the linearizing charts of the periodic orbits of $g$, as in Sec. \ref{sec-Buff}; we take one arc for each periodic point in $[x,G(x)]$, arcs are below attracting periodic points and above repelling periodic points. We assume that $\gamma'$ satisfies the following additional requirements (see Fig. \ref{fig-Buff-Rg}):
\begin{enumerate}
 \item $\gamma'$ joins $x$ to $G(x)$;
 \item $\gamma'$ is located between $\gamma$ and $g(\gamma)$ in $\bbC/\bbZ$.
\end{enumerate}
Then $\tilde \gamma:=\gamma'/G$ is a closed curve in $V/G$. Moreover, $\tilde \gamma$ can play the role of $\gamma$ in the construction of $\mathcal E(\mathcal R g)$, see Sec. \ref{sec-Buff}. Indeed, $\mathcal R g$ is induced by the maps $G^{n-1}-1, G^n-1$ on the quotient $V/G$, so has the same periodic points and linearizing charts as $G$; to satisfy the additional requirements above, we just need a suitable choice of $x_j$.

Let $\Omega\subset V/G$ be the annulus between $\tilde \gamma/G$ and $(\mathcal R g (\tilde \gamma))/G$. We conclude that $\mathcal E(\mathcal R g)$ is the quotient space of $\Omega$ by the action of $\mathcal R g$. Now we show that $\hat H$ (or rather its action on $V/G$) uniformizes $\mathcal E(\mathcal R g)$.

First, note that the lift of $\Omega$ to $V$ belongs to $\tilde \Pi^{(n+1)} \cap V$ because $g^n(\gamma')$ is between $g^n (\gamma)$ and $g^{n+1}(\gamma)$. Recall that the map $\hat H$ is defined in $\tilde \Pi^{(n+1)}$ and conjugates $G$ to the shift by $\tau$, so it descends to the map $\hat H \colon \Omega\to \bbC/\tau\bbZ$. Note that $\mathcal R g$ is induced by $G^{n-1}-1$, $G^n-1$ on the quotient $V/G$, and $\hat H$ conjugates  $G^{n-1}-1$, $G^n-1$ to the shifts by $(n-1)\tau-1, n\tau-1$. Thus $\hat H$ descends to the map from $\mathcal E(\mathcal R g)=\Omega/\mathcal R g$ to $\bbC/ \bbZ + \tau \bbZ$.

However its action on  the generators is non-trivial: in particular, $\hat H$ takes the curve $\tilde \gamma$ to
the generator $\tau \bbR/ \tau \bbZ$ of $\bbC/ \bbZ + \tau \bbZ$, because $\gamma'$ joins $x$ to $G(x)$.

Finally, $\mathcal E(\mathcal R g)$ is biholomorphically equivalent to $\bbC / \bbZ + \tau \bbZ $, with the first generator corresponding to $\tau\bbR/ \tau \bbZ$. We conclude that $ \bar \tau(\mathcal R f) \equiv \frac{-1}{\tau} \pmod 1$.
\end{proof}
\begin{proof}[The proof of Lemma \ref{lem-renorm-1}]
 The proof is analogous, with only the following modifications:
 \begin{itemize}
  \item We should consider the fundamental domain $[x, G^l(x)-k]$ instead of $[x, g(x)]$;
  \item $\Pi^{(n+1)}$ is the annulus between $\gamma$ and $g^M(\gamma)$, where $M$ is greater than the powers of $g$ that induce $\mathcal R^{k/l} g$;
  \item $\gamma'$ should join $x$ to $G^l(x)-k$ and be located between $\gamma$ and $g^l(\gamma)$;
  \item Consequently, $\hat H$ takes $\gamma'$ to the generator $\tau l -k$ of $\bbC/\bbZ+\tau \bbZ$.
 \end{itemize}
We get that $\mathcal E(\mathcal R ^{k/l}g)$ is biholomorphic to $\bbC/\bbZ+\tau \bbZ$, with the first generator corresponding to $\tau l -k$.  The lattice $\bbZ+\tau\bbZ$ is generated by  $\tau l -k$ and $-\tau r-s$, where $kr+ls=1$. So the modulus of $\mathcal E(\mathcal R ^{k/l}g)$ is $\bar \tau (\mathcal R ^{k/l}g) = \frac {\tau r+s}{-\tau l+k}$.
\end{proof}

\section{First-return maps for $k/l$-renormalizations}
\label{sec-proof-lem-kl}
In this section, we prove Lemma \ref{lem-renorm-kl}.
The following lemma reduces the question to the structure of orbits of $g$.
\begin{lemma}
\label{lem-firstreturn}
 Let $g$ be a circle diffeomorphism. Suppose that $g^b(x)$ is the first point of the orbit of $x$ that belongs to the arc $I=[x, g^{a}(x)]$, where $a<b$. Then the first-return map to $I$ under $g$ equals $g^b$ on $[x, g^{a-b}(x)]$ and $g^{b-a}$ on $[g^{a-b}(x), g^b(x)]$.
\end{lemma}

\begin{proof}
Our goal is to prove that for each point $y\in [x, g^{a-b}(x)]$, its image $g^b(y)\in [g^b(x), g^a(x)]$ is indeed the \emph{first} return to $I$; and for $y\in [g^{a-b}(x), g^a(x)]$, the first return is $g^{b-a}(y) \in [x, g^b(x)]$. Suppose that some point $y$ returns to $I$ earlier than that, say at $g^c(y)\in I$, $c<b$. Since $g^c$ takes the endpoints of $I$ to somewhere outside $I$ and $g^c(I)$ intersects $I$, we have $I\Subset g^c(I)$. Prove that this is not possible.

Indeed, in this case, $g^b(x) = g^c (z)$ for some $z\in I$, and $z$ is inside $I$.
Since $c<b$, we have $g^{b-c}(x) = z \in I$, thus  $x$ returns to the interior of  $I$ earlier than in $b$ iterates. A contradiction.
\end{proof}
Now the following lemma implies Lemma \ref{lem-renorm-kl}.
\begin{lemma}
\label{lem-firstpt}
 For any diffeomorphism $g$ with  $\rot g>k/l$ sufficiently close to $k/l$, the first point of the orbit of $x$ that belongs to $[x,g^l(x)]$ has the form $g^{Nl-r}(x)$.
\end{lemma}
\begin{proof}

First, we prove this for an irrational rotation.
For an irrational number $\alpha, 0<\alpha<1$, let its continued fraction representation be $\alpha =\frac{1}{a_1+\frac{1}{a_2+\dots}}$, and let $p_n/q_n$ be convergents of this continued fraction: $p_n/q_n = \frac{1}{a_1+\frac{1}{\dots+\frac {1}{a_n}}}$.
Suppose that  $\alpha\in\bbR\smallsetminus \bbQ$ is close enough to $k/l$ so that $k/l$ is a convergent to $\alpha$, $k/l=p_m/q_m$ for some $m$. Suppose that  $\alpha>k/l$. Let $T_{\alpha} \colon x \mapsto x+\alpha$ be an irrational rotation, $x_n := T^n_{\alpha}(x)$.

We will prove that the statement of the lemma holds true for $T_{\alpha}$ with  $N=a_{m+1}+2$.
We will use the following well-known facts about continued fractions.
\begin{enumerate}
\item \label{it-odd}The number $p_n/q_n$  approximates $\alpha$ from above for odd $n$, and from below for even $n$  (\cite[Theorem 4]{Kh}).
 \item \label{it-rec}$p_{n+1} = p_n a_{n+1}+p_{n-1}$, and $q_{n+1} = q_n a_{n+1}+q_{n-1}$  (\cite[Theorem 1]{Kh}).
\item \label{it-prod}$p_n q_{n-1}-p_{n-1}q_n=(-1)^{n+1}$ for all $n$ (\cite[Theorem 2]{Kh}).
 \item \label{it-close}For each $n$, $x_{q_n}$ is the point of a closest return to $x$: no points of the orbit $x_N, 0<N<q_n$, are closer to $x$ than $x_{q_n}$ (\cite[Theorem 17]{Kh}). Further, $x_{q_n}$ is to the left of $x$ if $n$ is odd, and to the right otherwise (this is clear from item \ref{it-odd}).
 \item \label{it-first} For any $n$, the first-return map on $[x_{q_{n}}, x_{q_{n+1}}]$ under $T_{\alpha}$ is $(T_{\alpha})^{q_{n+1}}$ on $[x, x_{q_n}]$ and $(T_{\alpha})^{q_{n}}$ on $[x, x_{q_{n+1}}]$. In particular, $x_{q_n+q_{n+1}}$ is the first point of the orbit of $x$ that belongs to $(x,x_{q_{n}})$.
\end{enumerate}
The last item is easily proved by induction.

We use all these facts for $n=m$. Due to item \ref{it-odd}, $m$ is even, because $\alpha>k/l = p_m/q_m$. Now item \ref{it-prod} implies $q_{m-1}k-p_{m-1}l=-1$. On the other hand, $rk+sl=1$ due to the definition of $k,l$, thus $(l-r)k-(k+s)l=-1$.

However there is the only number $a$ between $0$ and $l$ such that $ak \equiv -1 \pmod l$. Thus  $a=l-r=q_{m-1}$, and $k+s= p_{m-1}$. We conclude that $q_{m+1}= q_m a_m+q_{m-1}=l (a_m+1)-r$ due to item \ref{it-rec}, and item \ref{it-first} implies the statement of Lemma \ref{lem-firstpt} for $T_{\alpha}$.

Now, we prove this statement for arbitrary $g$. If $g$ has an irrational rotation number $\alpha$, its orbits are   ordered on the circle in the same way as orbits of $T_{\alpha}$, and the result follows.
Suppose that $g$ has rational rotation number $\rot g >k/l$ close to $k/l$, and $\rot g = \frac{1}{a_1+\frac{1}{\dots+\frac {1}{a_m+\frac{1}{a_{m+1}\dots}}}}$,  $p_m/q_m=k/l$ as before. Then we may perturb $g$ so that $a_{m+1}$ does not change, $\rot g$ becomes irrational, and the order of the first $(a_{m+1}+2)l-r$ points of the orbit of $x$ under $g$ does not change. This reduces the case of $\rot g \in \bbQ$ to the case $\rot g \notin \bbQ$, and completes the proof.

\end{proof}

\section{Continuity of complex rotation numbers (Lemma \ref{lem-tau-contin})}
\label{sec-contin}
 The proof of Lemma \ref{lem-tau-contin} is analogous to the proof of \cite[Lemma 5]{XB_NG}.
We split the proof into three lemmas. For all the lemmas below,  $g_k$ is a sequence  of analytic circle diffeomorphisms that tends to an analytic circle diffeomorphism $g$ uniformly on some neighborhood of $\bbR/\bbZ$. We will use repeatedly that
 $\rot g_{k} \to \rot g$ as $k\to \infty$.

\begin{lemma}
\label{lem-irr}
 Suppose that $\rot g$ is irrational.  Then $ \bar \tau (g_{k}) \to  \bar \tau (g)=\rot g$.
\end{lemma}

\begin{lemma}
\label{lem-parab}
 Suppose that $\rot g$ is rational and $g$ has parabolic cycles.

 (1) Suppose that $g$ has at most one parabolic cycle.  Then $ \bar \tau (g_{k}) \to  \bar \tau (g)=\rot g$.

 (2) Suppose that $g_{k} $ is a monotonic sequence. Then $ \bar \tau (g_{k}) \to  \bar \tau (g)=\rot g$.
\end{lemma}

\begin{lemma}
\label{lem-hyp}
 Suppose that $\rot g$ is rational, and $g$ is hyperbolic.  Then $ \bar \tau (g_{k}) \to  \bar \tau (g)$.
\end{lemma}
These lemmas constitute the proof of Lemma \ref{lem-tau-contin}. Also, Lemmas \ref{lem-parab} and \ref{lem-hyp} show that any bubble $B_{p/q, F_{\omega}}$ in a monotonic family is a continuous curve that starts and finishes at $p/q$.
\subsection{Irrational rotation numbers (Lemma \ref{lem-irr})}

 For the values of $k$ such that $\rot g_{k}$ is irrational, we have $ \bar \tau (g_{k})=\rot g_{k}$ which tends to $\rot g$, and the result follows.

 For the values of  $k$ such that $\rot g_{k}$ is rational, $\rot g_{k}= \frac{p_{k}}{q_{k}}$, the last two subcases in Theorem \ref{th-XB_NG} show that the distance between $ \bar \tau(g_{k})$ and  $\rot g_{k}$ is at most  $2\cdot D_{g_{k}}/(4\pi q_{k}^2)$. Since  $\rot g_{k} \to \rot g= \bar \tau(g)$, it suffices to prove that $ D_{g_{k}}/(4\pi q_{k}^2)$ tends to zero.

 Clearly, the denominator $q_{k}$ of $\rot g_{k}$ tends to infinity as $\rot g_{k}\to \rot g\notin \bbQ$. As for the distortion $D_{g_{k}} = \int_{\bbR/\bbZ} |\frac{g''_{k}}{g'_{k}}|$, we have $D_{g_{k}} \to D_g$ because $g_{k}$ tends to $g$ uniformly on some neighborhood of $\bbR/\bbZ$. So  $ D_{g_{k}}/(4\pi q_{k}^2)$ tends to zero, therefore $ \bar \tau (g_{k}) \to  \bar \tau (g)=\rot g$.

\subsection{Diffeomorphisms with parabolic cycles (Lemma \ref{lem-parab})}
 We split $g_k$ into  the following subsequences.
\begin{enumerate}
\item The subsequence with  $\rot g_k\neq \rot g$. The proof is literally the same as in the case of irrational rotation number (Lemma \ref{lem-irr} above).

\item The subsequence with  $\rot g_k= \rot g=p/q$ such that  $g_{k}$ has parabolic cycles. Then $\bar \tau(g_{k}) = \rot(g_{k}) = \rot (g) = \bar \tau(g)$ due to the second subcase in Theorem \ref{th-XB_NG}, and the result follows.

\item The subsequence with  $\rot g_k= \rot g=p/q$ such that  $g_{k}$ has only hyperbolic cycles on $\bbR/\bbZ$, and some of these cycles approach parabolic cycles of $g$. Then the multipliers of these real hyperbolic cycles of $g_{k}$ tend to $1$ as $k\to \infty$. The result follows from Lemma \ref{lem-disc-radius} (applied to $g_{k}$), because the radius $R$ tends to zero.

\item The subsequence with  $\rot g_k= \rot g=p/q$ such that  $g_{k}$ has only hyperbolic cycles on $\bbR/\bbZ$, and these cycles are detached from parabolic cycles of $g$. This means that the hyperbolic cycles of $g_k$ tend to the hyperbolic cycles of $g$, while all parabolic cycles of $g$ bifurcate and disappear from the real line.
\end{enumerate}
This is the only non-trivial case. The consideration is analogous to that in \cite[Lemma 15]{XB_NG}, where the case $g_{\eps}=g+\eps$, $\eps\to 0$, is considered. We are going to prove that $\bar \tau(g_{k}) \to \bar \tau(g)=p/q$.  First, we prove that $q\bar \tau(g_{k}) \to 0$ in $\bbH/\bbZ$.

In the notation of Sec. \ref{sec-Buff}, set
\[\tilde r_j := \frac{\log \phi_j^{-1}(x_j)}{\log \rho_j}\quad\text{and}\quad  \tilde s_j := \frac{\log |\phi_j^{-1}(x_{j-1})|}{\log \rho_j}+\frac{i \pi}{|\log \rho_j|}.\]
Put
\begin{equation}
\label{eq-sigma}
\sigma(f) := \sum_{j\in \bbZ/2mq Z}\tilde s_j-\tilde r_j.\end{equation}

We will use the following statement, see \cite[Lemma 12]{XB_NG}:
\begin{lemma}\label{lemma-qctwist}
For any hyperbolic circle diffeomorphism $f$ with rotation number $p/q$, we have  that
\[\dist_{\bbH/\bbZ}\left(q\bar \tau(f),-\frac{1}{\sigma(f)}\right)\leq 5D_f
\]
\end{lemma}
The proof of this lemma in \cite{XB_NG} contains an explicit construction of a quasiconformal homeomorphism between $\mathcal E(f)$ and the standard complex torus.

We apply this lemma to hyperbolic circle diffeomorphisms $g_{k}$. The corresponding notation for periodic points, their multipliers and linearizing charts of $g_{k}$ is $\alpha^{k}_j$, $\rho_j^{k}$, $\phi_j^{k}$; we  may and will assume that $x_j$ does not depend on $k$, and we define $\tilde r_j^{k}, \tilde s_j^{k}$, $\sigma(g_{k})$ as above.
The following lemma, together with Lemma \ref{lemma-qctwist}, implies $q\bar \tau(g_{k}) \to 0$.
\begin{lemma}
\label{lem-sigma}
 In assumptions of Lemma \ref{lem-parab} subcase (4), $\sigma(g_k)\to \infty$ as $k\to \infty$.
\end{lemma}
\begin{proof}
Note that the description of subcase (4) implies that $\alpha_j^{k}$ tend to all real hyperbolic periodic points of $g$. Let $\alpha_j:=\lim_{k\to\infty} \alpha_j^{k}$ be these periodic points; note that parabolic periodic points of $g$ are not in the list $\{\alpha_j\}$.

The multipliers $\rho_j^{k}$ of the real hyperbolic cycles of $g_{k}$ tend to the multipliers of the real hyperbolic cycles of $g$, so the imaginary parts of $\tilde s_j^k$ have finite limits. Prove that $\Re \sigma(g_{k})$ tends to infinity.

If an arc $[\alpha_j, \alpha_{j+1}]$ does not contain a parabolic periodic point of $g$, then $\phi_j^{k}$, $\phi_{j+1}^{k}$ tend to the linearizing charts of $g^q$ at $\alpha_j, \alpha_{j+1}$ on the whole arc $[\alpha_j, \alpha_{j+1}]$, and the real parts of $\tilde r_j^{k}, \tilde s_{j+1}^{k}$ have finite limits.

Let $[\alpha_j, \alpha_{j+1}]$ be one of the arcs that contain parabolic cycles of $g$. We have
\begin{equation}
\label{eq-r-s}
 \Re(\tilde s^{k}_{j+1}-\tilde r^{k}_{j}) = \frac{\log |(\phi_{j+1}^{k})^{-1}(x_{j})|}{\log \rho^{k}_{j+1}}-\frac{\log (\phi_j^{k})^{-1}(x_j)}{\log \rho_j^{k}}.
\end{equation}
The limit of this quantity is  $+\infty$ if $\alpha_j $ attracts and $\alpha_{j+1}$ repels; otherwise, the limit is $-\infty$. Indeed, the denominators stay bounded as mentioned above;   $\phi^{k}_j$, $\phi^{k}_{j+1}$ tend to linearizing charts of $\alpha_j, \alpha_{j+1}$ in small neighborhoods of $\alpha_j, \alpha_{j+1}$ respectively, however we need more and more iterates of $g_{k}$ to get to these neighborhoods from $x_j$. So either both $\log |(\phi_{j+1}^{k})^{-1}(x_{j})|$, $\log (\phi_{j}^{k})^{-1}(x_{j})$ tend to $+\infty$, or one of them remains bounded and the other one tends to $+\infty$. If the parabolic cycle that visits $[a_j, a_{j+1}]$ attracts from the right, i.e. $\alpha_j$ attracts and $\alpha_{j+1}$ repels, then $\log \rho^{k}_j<0<\log \rho^{k}_{j+1}$, and $\Re(\tilde s^{k}_{j+1}-\tilde r^{k}_{j})  \to +\infty$ as $k\to \infty$. If this parabolic cycle attracts from the left,  the limit is $-\infty$.

Now, $\sigma(g_k)$ is a sum of several bounded summands and several summands that tend to $+\infty$ ($-\infty$).
The proof is finished differently for the two parts of Lemma \ref{lem-parab}.

(1) $g_0$ has at most one parabolic cycle. If this cycle attracts from the left, all unbounded summands in the sum for $\sigma(g_{k})$ tend to $-\infty$, so $\sigma(g_{k})\to -\infty$. If this cycle attracts from the right, $\sigma(g_{k})\to +\infty$.

(2) $g_{k}$ is monotonic, say decreases with $k$. Recall that all parabolic orbits of $g$  disappear from the real line. Thus all of them attract from the left. So  $\sigma(g_k)$ is a sum of several bounded summands and several summands that tend to $-\infty$. Hence  $\sigma(g_{k})\to -\infty$.
\end{proof}

Lemmas \ref{lemma-qctwist} and \ref{lem-sigma} imply that $q\bar \tau(g_{k}) \to 0$ in $\bbH/\bbZ$, because $D_{g_k}$ is bounded, $D_{g_k}\to D_g$.
Note that $\bar \tau(g_{k})$ is in the disc of radius $D_{g_{k}}/(4\pi q^2)$ that is tangent to the real line at $p/q$ (see the last subcase of Theorem \ref{th-XB_NG-1}); so $\bar \tau(g_{k})\to p/q = \bar \tau(g)$, q.e.d.

\begin{remark}
\label{rem-discont}
 If $g$ has two periodic cycles, several  summands in \eqref{eq-sigma} for $g_{k}$ still tend to infinity, but may have different signs and may compensate each other. So the statement of Lemma \ref{lem-sigma} might be wrong, see the example below. This is the only place in the proof of Theorem \ref{th-main} where we use that ${\mathbf K}+c $ has at most one parabolic cycle.

 The following example  was suggested by Yu.Ilyashenko.
 Let $g_{\eps}=g_{v_{\eps}}^1$ be the time-one flow of the vector field $v_{\eps}(x)=\sin 2\pi x (\cos^2 2\pi x+\eps)$. Then $\overline \tau (g_0)=0$ because $g_0$ has parabolic fixed points at $1/4, 3/4$. To compute $\overline \tau (g_\eps)=\tau(\mathcal E(g_{\eps}))=\tau(\Pi^{\eps}/g_{\eps})$, we consider the complex time along $v_{\eps}$ as a new coordinate in the strip $\Pi^{\eps}$. In this coordinate, $\mathcal E(g_{\eps})$ becomes the standard torus with generators $I_{\eps}=\int_{\gamma} \frac {dx}{v_{\eps}(x)}$ and $1$. So its modulus is $1/I_{\eps}$. Here $\gamma$ is a curve from the construction of $\mathcal E(g_0)$ that passes above the repellor $0$ and below the attractor $0.5$ of $g_0$.

 One can compute the integral above and get $I_{\eps}= (V.P.)\int \frac {dx}{v_{\eps}(x)}-\pi i Res_0 \frac{1}{v_{\eps}} + \pi i Res_{0.5} \frac{1}{v_{\eps}} $, c.v.  the computation in  \cite[Lemma 3.4]{YuIM}. The first summand is zero due to the symmetry $v_{\eps}(x)=-v_{\eps}(-x)$. The second and the third summands tend to $-\pi i \frac 1{v'(0)}+\pi i \frac 1 {v'(0.5)} =-i$ as $\eps\to 0$. So $\lim_{\eps\to 0} \overline \tau(g_{\eps})=1/(-i)=i$, and $\overline \tau(g_0) =0$ as mentioned above. Therefore $\overline \tau(\cdot)$ is not continuous at $g_0 = g_{v_0}^1$.
\end{remark}

\subsection{Hyperbolic diffeomorphisms (Lemma \ref{lem-hyp})}
We are going to reduce the statement to the following lemma. Informally, it means that close gluings produce close complex tori.
\begin{lemma}
\label{lem-glue-contin}

 Let $A$ be the annulus bounded by two analytic essential curves $\gamma_{1,2}$ in $\bbC/\bbZ$. Let $G_0\colon \gamma_1 \to \gamma_2$ be an analytic diffeomorphism, let analytic maps $G_{k} \colon \gamma_1 \to \bbC/\bbZ$, tend to $G_0$ uniformly in a small neighborhood of $\gamma_1$ as $k\to\infty$. Let $A_{\delta}$ be a small neighborhood of $A$.

 Let $\tau_{k}\in\bbH/\bbZ$ be the moduli of the complex tori $A_{\delta}/G_{k}$, i.e. we suppose that there exists a biholomorphism that takes $A_{\delta}/G_{k}$ to $\bbC/\bbZ+\tau_{k}\bbZ$ and maps the class of $\gamma_1$ to the class of $\bbR/\bbZ$.

 Then $\lim_{k\to \infty} \tau_k =\tau_0$.
 \end{lemma}

\subsubsection{Reduction to Lemma \ref{lem-glue-contin}}
Choose the curve $\gamma\in \bbC/\bbZ$  as in Sec. \ref{sec-Buff}, and let $\Pi$ be the annulus between $\gamma$ and $g(\gamma)$. Let $\Pi_\delta$ be a neighborhood of $\Pi$ where both $g$ and $g_{k}$ are  univalent for large  $k$. Then $\mathcal E(g) = \Pi^{\delta}/g$ and $\mathcal E(g_{k})=\Pi^{\delta}/g_{k}$. So we may use Lemma \ref{lem-glue-contin} for $g, g_{k}$ playing the role of $G,G_{k}$. Formally, we may not take $\gamma_1=\gamma$ because $\gamma$ may be non-analytic; we should take $\gamma_1$ to be any analytic curve sufficiently close to $\gamma$ such that $\gamma_1\subset \Pi^{\delta}$ and $\gamma_2:=g(\gamma_1)\subset \Pi^{\delta}$.

Lemma \ref{lem-glue-contin} implies that $\bar \tau(g_{k})\to \bar \tau(g_{k})$ q.e.d.

\subsubsection{Proof of Lemma \ref{lem-glue-contin}}
Suppose that the annulus $A$ has modulus $\mu$. Let $A(\mu) : = \{z \in \bbC/\bbZ \mid 0 \le \Im z \le \mu \}$, then there exists a biholomorphism $\Phi \colon A(\mu) \to A$. Note that $\Phi$ extends analytically to a neighborhood of $A(\mu)$ because the boundaries of $A$ are analytic curves. Now letting $\hat G_{k}:=\Phi \circ G_{k} \circ \Phi^{-1}$ and $\hat G:=\Phi \circ G \circ \Phi^{-1}$ reduces the general case to the case $\gamma_1=\bbR/\bbZ$, $\gamma_2=\bbR/\bbZ+i\mu$. Below we only consider this case.

Our goal is to construct a quasiconformal homeomorphism $H$ that takes $A_\delta/G$ to $A_\delta/G_{k}$ and the class of $\bbR/\bbZ$ to itself, and to show that its quasiconformal dilatation $\left|\frac{H_{\bar z}}{H_z}\right|$ tends to zero uniformly in $A_\delta/G$ as $k\to \infty $.

Put $\xi (z)= G_{k}\circ G^{-1}$. For small $\delta$, for sufficiently large $k$, this map is well-defined in a $\delta$-neighborhood of $\bbR/\bbZ+i\mu$, and tends to identity uniformly within this neighborhood. Let $s\colon [0, \mu]\to [0,1]$ be a $C^2$-smooth monotonic map such that $s=1$ except for $[\mu-\delta, \mu+\delta]$ and $s=0$ in $(\mu-\delta/2, \mu+\delta)$. The estimate on $s'$ will depend on $\mu$ and $\delta$ only.  The required quasiconformal map is induced by
$$H(x,y) = s(y) (x+iy) + (1-s(y)) \xi(x+iy).$$
Indeed, $H$ induces the map between $A^{\delta}/G$ and $A^{\delta}/G_{k}$ because near the lower boundary of $A^{\delta}$, $H$ is identical, and near the upper boundary, it equals $\xi$. Outside the  $\delta$-neighborhood of  $\bbR/\bbZ+i\mu$, $H$ is identical and has quasiconformal dilatation equal to $0$. Inside this neighborhood, we have
\begin{align*}
H(x,y) &= \xi(x+iy)  + s(y) (x+iy - \xi(x+iy) );\\
\frac{\partial H}{\partial x} &= \xi'(x+iy) + s(y) (1-\xi'(x+iy));\\
  \frac{\partial H}{\partial y} &= i\xi'(x+iy)  + is(y) (1 - \xi'(x+iy))+ s'(y) (x+iy - \xi(x+iy)) ;\\
 2\frac{\partial H}{\partial \bar z} &=   \frac{\partial H}{\partial x}+i\frac{\partial H}{\partial y} = is'(y)(x+iy - \xi(x+iy)) ;\\
 2\frac{\partial H}{\partial z} &=    \frac{\partial H}{\partial x}-i\frac{\partial H}{\partial y} = 2\xi'(x+iy) + 2s(y) \cdot (1-\xi'(x+iy)) - is'(y)(x+iy - \xi(x+iy)).
\end{align*}
Note that $(x+iy - \xi(x+iy))$ tends to zero uniformly in the $\delta$-neighborhood of $\bbR/\bbZ+i\mu$ as $k\to \infty$, and the bound on $s'$ does not depend on $k$. So  $\frac{\partial H}{\partial \bar z}$ tends to zero uniformly in the strip $\Pi_{\delta}$.
This also shows that $\frac{\partial H}{\partial z}$ is uniformly close to $\xi'(x+iy) + s(y) \cdot (1-\xi'(x+iy))$ which is between $1$ and $\xi'(x+iy)$. Since $\xi'(x+iy)$ uniformly tends to $1$, we conclude that $\frac{\partial H}{\partial z}$ is uniformly close to $1$.

Finally, the quasiconformal dilatation of $H$ is uniformly close to $0$. This implies that $H$ is a homeomorphism, and shows that the modulus of $A^{\delta}/G_{k}$ tends to the modulus of  $A^{\delta}/G_{k}$  as $k\to \infty$.

% Note also that $\xi'(x+iy)$ tends uniformly to $1$ in the $\delta$-neighborhood of $\bbR/\bbZ+ib$, and outside this neighborhood, $s(y)$ is zero. So  $\frac{\partial H}{\partial z}$ is uniformly close to $2\xi'(x+iy)\chi_b$

\printbibliography
\end{document}